\documentclass{article}
\usepackage{amsmath,amsthm,amssymb,mathtools,stmaryrd}
\usepackage[all]{xy}
\usepackage{color}
\input{grcawol.sty}
\input{grcalcx.sty}
\input{grcahbox.sty}

\renewcommand\theta{\vartheta}
\usepackage{verbatim}
\newlength{\itemwidth}
\numberwithin{equation}{section}

\newtheorem{Lem}{Lemma}[section]
\newtheorem{Prop}[Lem]{Proposition}
\newtheorem{Cor}[Lem]{Corollary}
\newtheorem{Thm}[Lem]{Theorem}
\newtheorem{LemDef}[Lem]{Lemma and Definition}
\theoremstyle{definition}
\newtheorem{Def}[Lem]{Definition}
\theoremstyle{remark}
\newtheorem{Rem}[Lem]{Remark}
\newtheorem{Expl}[Lem]{Example}

\renewcommand\o{\otimes}
\newcommand{\ob}{\overline{\bullet}}
\newcommand{\ot}{\otimes}
\DeclareMathOperator\id{\operatorname{id}}
\newcommand\op{{\operatorname{op}}}

\newcommand\ol{\overline}
\newcommand\inv{^{-1}}

\newcommand\Hom{\operatorname{Hom}}

\renewcommand\epsilon\varepsilon

\newcommand\B{\mathcal B}
\newcommand\BB{\overline{B}}

\newcommand{\g}{\mathfrak{g}}
\newcommand{\cG}{\mathcal{G}}
\newcommand{\cH}{\mathcal{H}}
\newcommand{\h}{\mathfrak{h}}
\newcommand\CL{\mathcal L}
\newcommand\cL{\mathcal L}
\newcommand\cI{\mathcal I}
\newcommand\cJ{\mathcal J}
\newcommand\cM{\mathcal M}
\newcommand\cO{\mathcal O}
\newcommand\cU{\mathcal U}
\newcommand\cX{\mathfrak X}

\newcommand\N{\mathbb N}

\def\BHM#1.#2.#3.#4.{{^{#1}_{#3}\mathcal B^{#2}_{#4}}}
\newcommand\rcofix[2]{{#1}{^{\operatorname{co} #2}}}

 \makeatletter
\def\namelabel#1#2{\@bsphack
  \protected@write\@auxout{}%
         {\string\newlabel{#1.nme}{{#2}{#2}}}%
  \@esphack}
\def\nmlabel#1#2{\label{#2}\namelabel{#2}{#1}}

\makeatother

\newcommand{\la}{{\triangleright}}
\newcommand{\ra}{{\triangleleft}}
\newcommand{\bla}{{\blacktriangleright}}

\newcommand{\can}{\textup{can}}
\newcommand\skp[2]{\langle #1|#2\rangle}

\newcommand{\z}{{}_{\scriptscriptstyle{(0)}}}

\newcommand{\on}{{}_{\scriptscriptstyle{(1)}}}

\newcommand{\mo}{{}_{\scriptscriptstyle{(-1)}}}

\renewcommand{\t}{{}_{\scriptscriptstyle{(2)}}}

\renewcommand{\th}{{}_{\scriptscriptstyle{(3)}}}

\newcommand{\fo}{{}_{\scriptscriptstyle{(4)}}}
\newcommand{\fiv}{{}_{\scriptscriptstyle{(5)}}}
\newcommand{\si}{{}_{\scriptscriptstyle{(6)}}}

\newcommand{\tuno}[1]{{#1}{}{}^{\scriptscriptstyle{<1>}}}
\newcommand{\tdue}[1]{{#1}{}{}^{\scriptscriptstyle{<2>}}}

\title{Braided Hopf algebroids and Lie algebroids}
\author{Xiao Han\\
}
\date{}
\begin{document}

\maketitle

\begin{abstract}We construct  braided Hopf algebroids in a braided monoidal category over a field $k$ and study its properties using graphical representation. Then we study Ehresmann-Schauenburg Hopf algebroids assocaited to braided Hopf Galois extensions. 

We introduce braided Lie-Rinehart algebras which generalize the definition in \cite{ALP24, ALP23, ALP25} and study its universal enveloping algebra under symmetric condition. In particular, we introduce Lie-Rinehart algebras (braided Lie-algebroids) associated with braided Hopf algebroids in the category of the module of a triangular Hopf algebra. We then focus on the case for  braided Ehresmann-Schauenburg Hopf algebroids and study their right invariant vector fields, which turn out to be isomorphic to the  $H$-equivariant vector fields on the quantum principal bundle (Hopf Galois extension) as Lie-Rinehart algebras.

Finally, we introduce braided jet Hopf algebroids associate to braided Hopf algebroids in the case that the source and target subalgebra belong to the braided center. Moreover, we show there is a dual pairing between the $k$-order universal enveloping algebra of the braided right invariant vector fields  and the $k$-th jet space of a braided Hopf algebroid and further more a skew pairing between the universal enveloping algebra of its braided right invariant vector fields  and its jet Hopf algebroid under a finiteness condition.
\end{abstract}
\section{Introduction}

Motivated by the geometric picture that given a Lie group, its right invariant vector fields and the tangent space of the unit is a Lie algebra, Woronowicz \cite{SLW} had invented a quantum (braided) Lie algebra associated with a Hopf algebra $H$ in terms of the dual space of its first order differential calculus in the category of Yetter-Drinfeld modules of $H$. Similarly, a Lie-Rinehart algebra of a Lie groupoid can be viewed as right (or left) invariant vector fields $\cX$ which are tangent to the target fiber. Equivalently, it can also be viewed as the restriction of $\cX$  on the unit submanifold which are also sections on the vector bundle of the associated Lie algebroid. However, for the quantization of Lie-Rinehart algebroids, one can not use the same method as Woronowicz \cite{SLW}, because to construct a braided Lie algebroid we need a braiding over a field for the construction of a braided Lie bracket.  However, the Yetter-Drinfeld module of a Hopf algebroid is always over a noncommutative algebra \cite{X. Han23}. Therefore, we study braided Hopf algebroids in a braided category over a field $k$.

There are different definitions for Hopf algebroids \cite{Boehm,schau1}, here we are going to study the braided Hopf algebroids in the sense of \cite{schau1}. Namely, a bialgebroid without an antipode, but admits  bijective Galois and anti-Galois maps. Following Takeuchi's construction \cite{Takeuchi77}, we introduce an analog notation for the braided balanced tensor product and the braided Takeuchi product in order to make the corresponding coproduct an algebra map when co-restricting to the Takeuchi product. Furthermore, we study the properties of Galois map and anti-Galois map by using  graphical computations. After that, we study skew pairing as in \cite{schau1} for the further use for the universal enveloping algebra and the jet Hopf algebroid of a braided Hopf algebroid. In particular, we study braided Hopf Galois extensions and their Ehresmann-Schauenburg Hopf algebroids \cite{schau4} as a generalization of the study on braided Galois objects \cite{schau2}. Also, we show the Ehresmann-Schauenburg Hopf algebroid of a braided Hopf Galois extension is a braided Hopf algebroid with an explicit graphical description that matches the classical result in \cite{HS25}.

Next, we introduce braided Lie-Rinehart algebra in a braided category as a generalization of \cite{ALP24,ALP23, ALP25}. We study their universal enveloping algebra by using a similar construction of \cite{BKS24,ES19}. However, we need the braided category being symmetric in order to make the ideal a coideal for the Hopf algebroid, because only in this case we can construct a well defined braided universal enveloping algebra as a braided Hopf algebroid. Motivated by the geometric picture that the right invariant vector fields on a Lie groupoid form a Lie-Rinehart algebra, we study the right invariant vector fields on quantum groupoids (Hopf algebroids) in the category of the left modules of a triangular Hopf algebra $(K,R)$. Namely, a right invariant vector field is a braided derivative on a braided Hopf algebroid that vanishes on the target subalgebra. Also, we study its composition with the counit, which can be viewed as  the differentiation of the bisection \cite{HLL23,HM23}. Geometrically, it corresponds to its restriction on the unit submanifold of the base. In particular, we focus on the study of the right invariant vector fields of Ehresmann-Schauenburg Hopf algebroid $L(P,H)$ associated to a braided $H$-Galois extension $B\subset P$. We show that the braided  right $H$-equivariant vector fields $\textup{Der}^{H}(P)$ \cite{ALP24,ALP23,ALP25} on the quantum principal bundle is isomorphic to the braided right invariant vector fields $\cX(L(P,H))$ on $L(P,H)$ as Lie-Rinehart algebras. Moreover, we have the short exact sequence of Lie-Rinehart algebras in ${}_{K}\cM$:
  \[\xymatrix{0\ar[r]&\cX_{\textup{ver}}(L(P,H))\ar[r]^-{\iota}\ar[d]_{}&\cX(L(P,H))\ar[r]^{\rho}\ar[d]_{} &\textup{Der}(B)\ar[d]_{=}\ar[r]&0\\
0\ar[r]&\textup{aut}_{B}(P)\ar[r]^-{\iota} &\textup{Der}^{H}(P)\ar[r]^{\rho}&\textup{Der}(B)\ar[r]&0}.\]

Finally, we construct a braided jet Hopf algebroid of a braided Hopf algebroid $\cL$ under the condition that the source and target subalgebra belongs to the braided center of $\cL$ which generalizes \cite{HM25} where  the jet Hopf algebroid of a pair Hopf algebroid $B\ot\BB$ in the commutative case was studied.  Further more, we show the $k$-order differential operators $\cU_{B}(\cX(\cL))_{\leq k}$ generated by the Lie-Rinehart algebra of $\cL$ are $B$-linear maps of the $k$-th jet space $\cJ^{k}(\cL)$ which match the  geometric case. While in \cite{CG}, the $k$-jet space is defined to be the dual space of  $\cU_{B}(\g)_{\leq k}$. Further more, under some finiteness condition, we can show there is a skew pairing between the universal enveloping algebra $\cU_{B}(\cX(\cL))$ and the jet Hopf algebroid of $\cL$.

\subsection*{Acknowledgements}
The author was supported by Leverhulme Trust project grant RPG-2024-177. The author thanks Peter Schauenburg and Shahn Majid for helpful discussions. In particular,  the author thanks Mathieu ~Sti\'enon for offering the geometric idea on the jet bundles of Lie groupoids. Moreover, the author thanks Peter Schauenburg for sharing his Latex  source code for drawing the diagrams.

\section{Preliminaries}\nmlabel{Section}{sec:prelims}
Throughout the paper, $\B$ denotes a strict braided monoidal category over a field $k$ with
equalizers, coequalizers and a braiding $\tau$. We will use
graphical calculus to do computations in $\B$, using the notations
$$\tau_{VW}=\gbeg23\got1V\got1W\gnl\gbr\gnl\gob1W\gob1V\gend\qquad
  \text{ and }\qquad
  \tau_{VW}\inv=\gbeg23\got1W\got1V\gnl\gibr\gnl\gob1V\gob1W\gend$$
for the braiding. On elements we denote the image of the braiding by $\tau_{VW}(v\otimes w)=w_{\alpha}\otimes v^{\alpha}$ (or any other Greek letter) for $v\otimes w\in V\ot W$, and $\tau^{-1}_{VW}(w\ot v)=v^{\ol \alpha}\ot w_{\ol \alpha}$. Hence $v^{\alpha\ol \beta}\ot w_{\alpha \ol \beta}=v\ot w$. We call $\B$ symmetric if $\tau_{WV}=\tau^{-1}_{VW}$ (i.e. $v_{\alpha}\ot w^{\alpha}=v^{\ol \alpha}\ot w_{\ol \alpha}$). We will use
$$\nabla_A=\gbeg 23\got1A\got1A\gnl\gmu\gnl\gob2A\gend,
  \qquad
  \eta_A=\gbeg13\gnl\gu1\gnl\gob1A\gend,
  \qquad
  \mu=\mu_r=\gbeg 23\got1M\got1A\gnl\grm\gnl\gob1M\gend,
  \qquad\text{and}\qquad
  \mu_\ell=\gbeg23\got1A\got1M\gnl\glm\gnl\gvac1\gob1M\gend
$$
for the multiplication and unit of an algebra $A$ in $\B$, the
module structure of a right $A$-module $M\in\B_A$, and the module
structure of a left module $M\in{_A\B}$. We will denote the right (resp. left) action of $A$ on  $M$ by $m\ra a$ or $ma$ (resp. $a\la m$ or $am$).
The opposite algebra of $A$  (denoted by $\ol A$ or $A^{\op}$) is defined on the same underlying object, with the multiplication determined by
$$\gbeg 24
 \got 1{A}\got 1{A}\gnl
  \gmp{\overline{\bullet}}\gmp{\overline{\bullet}}\gnl
  \gmu\gnl
  \gob 2{\ol A}
  \gend= \gbeg 35
  \got 1A\gvac 1\got 1A\gnl
  \gbbr31\gnl
  \gwmu 3\gnl
  \gvac 1\gmp{\overline{\bullet}}\gnl
  \gvac 1\gob 1{\ol A}
  \gend \qquad\text{or
}\qquad\gbeg 25
  \got 1{A}\got 1{A}\gnl
  \gmp{\overline{\bullet}}\gmp{\overline{\bullet}}\gnl
  \gibr\gnl
  \gmu\gnl
  \gob 2{\ol A}
  \gend= \gbeg 34
  \got 1A\gvac 1\got 1A\gnl
  \gwmu 3\gnl
  \gvac 1\gmp{\overline{\bullet}}\gnl
  \gob 3{\ol A}
  \gend,$$  
where $\overline{\bullet}:A\to \ol A, a\mapsto \ol a$ is the canonical map. More precisely, $\ol a \ol b=\overline{\nabla_{A}\circ \tau_{AA}(a\otimes b)}$. An algebra $A$ is called braided commutative, if it is equal to its opposite algebra in $\B$, i.e. $\nabla_{A}=\nabla_{A}\circ \tau_{AA}$. Clearly, there is `another' opposite algebra product defined by composing the braiding inverse with the multiplication. But in this paper, we will only use the opposite algebra defined above. Moreover, we define $\tau_{\ol{A}\,\ol{A}}:=\tau_{AA}^{-1}$. As a result, $\ol{\ol{A}}=A$.  Let $A$ and $B$ be two algebras in $\B$. Then $A\otimes B$ is an algebra in $\B$ with product and unit given by
$$\nabla_{A\otimes B}=\gbeg64\got1A\got1B\got1A\got1B\gnl
\gcl1\gbr\gcl1\gnl
\gmu\gmu\gnl
\gob2A\gob2B
\gend,\qquad
\eta_{A\otimes B}=\gbeg23\gnl
\gu1\gu1\gnl
\gob1A\gob1B
\gend.
$$

\begin{Prop}\label{prop. module of tensor algebra}
Let $A$ and $B$ be two algebras in $\B$. Then $P$ is a left $A\otimes B$-module if and only if $P$ is a left $A$ and $B$-module such that
\[b(ap)=a_{\alpha}\,(b^{\alpha}p)\]
for all $a\in A, b\in B$ and $p\in P$. Similarly, $P$ is a right $A\otimes B$-module if and only if $P$ is a right $A$ and $B$-module such that
\[(pb)a=(pa_{\alpha})b^{\alpha}.\]
\end{Prop}

Let $M$ be a left $\ol B$-module and $N$ be left $B$-module.We define 
\begin{align*}
M\diamond_{B} N:=\int_{b} {}_{\Bar{b}}M\ot {}_{b}N:=M\otimes N/\langle \Bar{b}m\otimes n- m_{\alpha}\otimes b^{\alpha}n|b\in B, m\in M, n\in N\rangle.
\end{align*}
If $M$ is a right $B$-module and $N$ is a left $B$-module, then we define 
\begin{align*}
M\ot_{B} N:=\int_{b} M_{b}\ot {}_{b}N:=M\otimes N/\langle mb\otimes n- m\otimes bn|b\in B, m\in M, n\in N\rangle.
\end{align*}
Similarly, if $M$ is a right $\BB$-module and $N$ is a left $\BB$-module, then we define 
\begin{align*}
M\ot_{\BB} N:=\int_{b} M_{\Bar{b}}\ot {}_{\Bar{b}}N:=M\otimes N/\langle m\Bar{b}\otimes n- m\otimes \Bar{b}n|b\in \BB, m\in M, n\in N\rangle.
\end{align*}
Moreover, if $M$ is a right $\BB$-module and $N$ is a right $B$-module, we define
\begin{align*}
    \int^{b}M_{\Bar{b}}\ot N_{b}:=\{\sum_{i}m_{i}\ot n_{i}\in M\ot N\quad |\quad \sum_{i}m_{i}\overline{b_{\alpha}}\ot n_{i}^{\alpha}=m_{i}\ot n_{i}b, \forall b\in B\}.
\end{align*}

% The symbol $\int^{b}$ and $\int^{c}$ braided commute, also $\int_{b}$ and $\int_{c}$ braided commute. More precisely, given a right $A\ot B$-module $M$ and a left $A\ot B$-module $N$
% \[\int_{a}\int_{b}M_{a,b} \ot {}_{b,a}N=\int_{b}\int_{a}M_{b,a} \ot {}_{a,b}N\]
 For any $\BB$-bimodule $M$ and any $B$-bimodule $N$, we also define
\begin{align*}
    M\times_{B}N:=\{\sum_{i}m_{i}\ot n_{i}\in M\diamond_{B} N\quad |\quad \sum_{i}m_{i}\overline{b_{\alpha}}\ot n_{i}^{\alpha}=m_{i}\ot n_{i}b, \forall b\in B\}.
\end{align*}
$M\times_{B}N$ is called the Takeuchi product \cite{Takeuchi77} of $M$ and $N$  in $\B$. We can check that it is well defined. Indeed, the right $B$-module structure on $M\diamond_{B}N$ given by $(m\ot n) b=m\otimes nb$ is clearly well defined. To check the right $\BB$-module structure on $M\diamond_{B}N$ given by $(m\ot n)\ol b=m\ol{b_\alpha}\otimes n^{\alpha}$ is well defined, we have
\begin{align*}
\ol am\,\ol{b_{\alpha}}\ot n^{\alpha}=(m\,\ol{b_{\alpha}})_{\beta}\ot a^{\beta}\,n^{\alpha}=m_{\gamma}\,\ol{b_{\alpha\beta}}\ot a^{\gamma\beta}\,n^{\alpha}=m_{\beta}\ol{b_{\alpha} }\ot (a^{\beta}n)^{\alpha}.
\end{align*}
Actually, it is more obvious to give the graphical computations
$$\gbeg5{6}\got1B\got1M\got1{\diamond_{B}}\got1N\got1B\gnl
\gmp{\ob}\gcl1\gvac1\gcl1\gmp{\ob}\gnl
\glm\gvac1\gbr\gnl
\gvac1\gcn2113\gcl1\gcl1\gnl
\gvac1\gvac1\grm\gcl1\gnl
\gvac1\gvac1\got1{M}\got1{\diamond_{B}}\got1N
\gend
=\gbeg5{8}\got1B\got1M\got1{\diamond_{B}}\got1N\got1B\gnl
\gmp{\ob}\gcl1\gvac1\gcl1\gmp{\ob}\gnl
\gcl1\gcl1\gvac1\gbr\gnl
\gcl1\gcn2113\gcl1\gcl1\gnl
\gcn2113\grm\gcl1\gnl
\gvac1\glm\gvac1\gcl1\gnl
\gvac1\gvac1\got1{M}\got1{\diamond_{B}}\got1N
\gend=\gbeg5{8}\got1B\got1M\got1{\diamond_{B}}\got1N\got1B\gnl
\gcl1\gcl1\gvac1\gcl1\gmp{\ob}\gnl
\gcl1\gcl1\gvac1\gbbr22\gnl
\gcn2113\gcn21{-1}1\gnl
\gvac1\gcl1\grm\gcl1\gnl
\gvac1\gbr\gcn2131\gnl
\gvac1\gcl1\glm\gnl
\gvac1\got1{M}\got1{\diamond_{B}}\got1N
\gend
=\gbeg7{8}\got1B\got1M\got1{\diamond_{B}}\got1N\got1B\gnl
\gcl1\gcl1\gvac1\gcl1\gmp{\ob}\gnl
\gbr\gvac1\gcl1\gcl1\gnl
\gcl1\gcn2113\gcl1\gcl1\gnl
\gcl1\gvac1\glm\gcl1\gnl
\gcn2115\gvac1\gbr\gnl
\gvac1\gvac1\grm\gcl1\gnl
\gvac1\gvac1\got1{M}\got1{\diamond_{B}}\got1N
\gend.
$$
We also denote Takeuchi product by $\int^{a}\int_{b} {_{\Bar{b}}M_{\Bar{a}}}\ot {_bN_a}.$

Let $B^e:=B\ot\BB$. If $P$ is a $B^{e}$-bimodule, then $P\times_{B}N$ is a $B$-bimodule with $B$ acting on $P$ in a braided way. More precisely, $a(p\otimes n)b=a\,p\,b_{\alpha}\otimes n^{\alpha}$.  Similarly, $M\times_{B}P$ is a $\BB$-bimodule given by $\Bar{a}(n\otimes p)\Bar{b}=n_{\alpha}\ot \overline{a^{\alpha}}\, p\,\Bar{b}$. If both $M$ and $N$ are $B^{e}$-bimodules, then $M\times_{B} N$ is also a $B^{e}$-bimodule. 

However, the product $\times_{B}$ is neither associative nor unital on the category of $B^{e}$-bimodules. Given three $B^{e}$-bimodules $M,N, P$, we can define
\begin{align*}
    M\times_{B}P\times_{B}N:=\int^{a,b}\int_{c,d} {}_{\Bar{c}}M_{\Bar{a}}\ot {}_{c,\Bar{d}}P_{a, \Bar{b}}\ot {}_{d}N_{b},
\end{align*}
where $\int^{a,b}:=\int^{a}\int^{b}$ and $\int_{c,d}:=\int_{c}\int_{d}$. Notice that by Proposition \ref{prop. module of tensor algebra}, we have $\int^{a}\int^{b}=\int^{b}\int^{a}$ and $\int_{c}\int_{d}=\int_{d}\int_{c}$ as the $B$ and $\BB$-module structures on $P$ braided commute. There are maps
\begin{align*}
    &\alpha:(M\times_{B}P)\times_{B} N\to M\times_{B}P\times_{B}N,\quad m\ot p\ot n\mapsto m\ot p\ot n,\\
    &\alpha':M\times_{B}(P\times_{B} N)\to M\times_{B}P\times_{B}N,\quad m\ot p\ot n\mapsto m\ot p\ot n.\\
\end{align*}
Notice that neither $\alpha$ nor $\alpha'$ are isomorphisms  in general. In particular, if all $B$-module and $\BB$-module structures are faithfully flat, then $\alpha$ and $\alpha'$ are isomorphisms. 

\section{Braided Hopf algebroids}
Given an algebra $B\in\B$, a \textit{$B$-ring} is an algebra $P$ in $\B$ if there is an algebra map $\eta:B\to P$ in $\B$. We also denote the map graphically by without mention the map $\eta$
$$\eta=\gbeg1{3}\got1B\gnl
\gcl1\gnl\gob1{P}\gend:=\gbeg1{5}\got1B\gnl
\gcl1\gnl\gmp{\eta}\gnl\gcl1\gnl\gob1{P}\gend$$
If $B\subset P$ is a subalgebra, we call $B$ belongs to the braided center of $P$ if $bp=p_{\alpha}b^{\alpha}$ for all $b\in B$ and $p\in P$.
\begin{Lem}
If $P$ is a $B\otimes \BB$-ring, then $P\times_{B} P$ is a $B\otimes \BB$-ring with structure 
\[(p\ot q)(p'\ot q')=pp'_{\alpha}\ot q^{\alpha}q',\qquad\eta(a\otimes \ol b)=a\otimes \ol b\in P\times_{B}P.\]
\end{Lem}
\begin{proof}
First, we show that the map $(P\diamond_{B}P)\times (P\ot P)\to P\diamond_{B}P, (p\ot q)\ot (p'\ot q')\mapsto pp'_{\alpha}\ot q^{\alpha}q'$ is well defined. Indeed,
\begin{align*}
p_{\gamma}p'_{\alpha}\ot (b^{\gamma}q)^{\alpha}q'=p_{\gamma}p'_{\alpha\beta}\ot b^{\gamma\beta}\,q^{\alpha}q'=(pp'_{\alpha})_{\beta}\ot b^{\beta}\,q^{\alpha}q'=\ol{b}pp'{}_{\alpha}\ot q^{\alpha}q'.
\end{align*}
Second, we show that  the map $(P\times_{B}P)\times (P\diamond_{B} P)\to P\diamond_{B}P, (p\ot q)\ot (p'\ot q')\mapsto pp'_{\alpha}\ot q^{\alpha}q'$ is well defined. Indeed, by the same method as above, the map factors the first diamond product in $P\diamond_{B}P$. To see the map factors through the second diamond product, we have
\begin{align*}
p(\ol{b}p')_{\alpha}\ot q^{\alpha}q'=p\ol{b_{\alpha}}p'{}_{\beta}\ot q^{\alpha\beta}q'=pp'{}_{\beta}\ot (qb)^{\beta}q'=pp'{}_{\alpha\beta}\ot q^{\beta}b^{\alpha}q'.
\end{align*}
Finally, we check that the restriction of the map in $(P\times_{B}P)\times (P\times_{B} P)$ has image in $P\times_{B} P$:
\begin{align*}
pp'{}_{\alpha}\ol{b}_{\beta}\ot (q^{\alpha}q')^{\beta}=pp'{}_{\alpha}\ol{b}_{\beta\gamma}\ot q^{\alpha\gamma}q'^{\beta}=p(p'\ol{b}_{\beta})_{\alpha}\ot q^{\alpha}q'^{\beta}=pp'_{\alpha}\ot q^{\alpha}q'b.
\end{align*}
It is not hard to see $\eta$ is an algebra map.
\end{proof}
In particular,  $P\times_{B} P$ and $B$ are left $B^e$-modules with left $B^e$-actions given by 
\begin{equation}\label{eq:rbgd.bimod}
(a\ot \ol b)(p\ot q)=ap_{\alpha}\ot \ol{b^{\alpha}}q,\qquad (a\ot \ol {a'})b=ab_{\alpha}\,a'{}^{\alpha}.
\end{equation}
For every $B^e$-ring $P$, we will denote $s:=\eta|_{B\ot 1}:B\to P$ and $t:=\eta|_{1\ot \ol B}:\ol B\to P$. If $P$ is a $B^e$-ring, we also denote the source and target map graphically by
$$s=\gbeg1{3}\got1B\gnl
\gcl1\gnl\gob1{P}\gend\qquad\textup{and}\qquad t=\gbeg1{3}\got1{\BB}\gnl
\gcl1\gnl\gob1{P}\gend$$
without mentioning $s$ and $t$ explicitly.
\begin{Def}
Given an algebra $B\in\B$, a left $B$-bialgebroid is a $B\ot \BB$-ring $\CL$ (with $\eta:B\ot \BB\to \CL$), such that
\begin{itemize}
\item[(i)]  There are two left $B^{e}$-module maps in the sense of (\ref{eq:rbgd.bimod}), the coproduct $\Delta:\CL\to \CL\times_{B} \cL$ and counit $\varepsilon:\CL\to B$ satisfying
\begin{align*}
&\alpha\circ(\Delta\times_{B}\id)\circ\Delta=\alpha'\circ(\id\times_{B}\Delta)\circ\Delta:\cL\to \CL\times_{B}\cL\times_{B}\CL\\    &(\varepsilon\diamond_{B} \id)\circ \Delta=(\id\diamond_{B} \varepsilon)\circ \Delta=\id,
\end{align*}
for any $X\in\CL$.  We will use the sumless Sweedler index to denote the coproduct, namely, $\Delta(X)=X\on\ot X\t$.
\item[(ii)] The coproduct is an algebra map. The counit $\varepsilon$ is a left character in the following sense:
\begin{equation*}\varepsilon(1_{\CL})=1_{B}, \quad \varepsilon(X\varepsilon(Y))=\varepsilon(XY)=\varepsilon(X\overline{\varepsilon(Y)})\end{equation*}
for all $X,Y\in \cL$ and $a\in B$.
\end{itemize}
A left $B$-bialgebroid is a left Hopf algebroid if the map
\[\lambda:\CL\ot_{\BB}\CL\to \cL\diamond_{B}\cL,\qquad X\otimes Y\mapsto X\on\ot X\t\,Y\]
is bijective. A left $B$-bialgebroid is an  anti-left Hopf algebroid if the map
\[\mu:\CL\ot_{B}\CL\to \cL\diamond_{B}\cL,\qquad X\otimes Y\mapsto X\on\,Y_{\alpha}\ot X\t{}^{\alpha}.\]
A left $B$-bialgebroid is a $B$-Hopf algebroid if it is a left Hopf algebroid and anti-left Hopf algebroid.
\end{Def}

\begin{Rem}
We can see that the above definition generalizes the definition of Hopf algebroids in (\cite{schau1}, Thm and Def 3.5.). In the following, we will use
$$\Delta=\gbeg23\gvac1\got1{\CL}\gnl\gwcm3\gnl \gvac{1}\gob 1{\CL\times_{B}\CL}\gend
  \qquad,
  \varepsilon=\gbeg1{5}\got1{\CL}\gnl\gcl1\gnl\gmp{\varepsilon}\gnl\gcl1\gnl
  \gob1B\gend
$$
for the comultiplication and counit of a bialgebroid $\CL$ in $\B$.
% the comodule structure of a right $C$-comodule $M\in\B^C$, and
% that of a left $C$-comodule $M\in{^C\B}$. 
Graphically, the coproduct is an algebra map means  
$$
 \gbeg3{5} 
  \got 1{\CL}\gvac 1\got 1{\CL}\gnl
  \gwmu 3\gnl
  \gvac1\gcl1\gnl
\gwcm3\gnl
 \gvac{1}\gob 1{\CL\times_{B}\CL}
\gend
=\gbeg4{5}
 \got2{\CL}\got2{\CL}\gnl
 \gcmu\gcmu\gnl
\gcl1\gbr\gcl1\gnl
\gmu\gmu\gnl
\gvac1\gob2{\CL\times_{B}\CL}
\gend.$$
The counital property means
$$
 \gbeg2{5}
 \got2{\CL}\gnl
 \gcmu\gnl
\gmp{\varepsilon}\gcl1\gnl
\glm\gnl
\gvac1\gob1{\CL}
\gend=\gbeg1{4}\got1{\CL}\gnl
\gcl1\gnl
\gcl1\gnl
\gob1{\CL}
\gend
= \gbeg2{6}
 \got2{\CL}\gnl
 \gcmu\gnl
\gcl1\gmp{\ol{\varepsilon}}\gnl
\gibr\gnl
\glm\gnl
\gvac1\gob1{\CL}
\gend.$$
where $\ol{\varepsilon}:=(\ob)\circ \varepsilon:\CL\to \BB$.  Finally, the canonical maps $\lambda$ and $\mu$ mean
$$\lambda=\gbeg4{4}\got3{\CL\otimes_{\BB}\CL}\gnl
\gcmu\gcl1\gnl
\gcl1\gmu\gnl
\gob2{{\CL\diamond_{B}\CL}}
\gend,\qquad\mu=\gbeg4{5}\got3{\CL\otimes_{B}\CL}\gnl
\gcmu\gcl1\gnl
\gcl1\gbr\gnl
\gmu\gcl1\gnl
\gob3{\CL\diamond_{B}\CL}
\gend.$$
We can check that $\mu$ is well defined. Indeed,
\begin{align*}
\mu(X\,b\ot Y)=X\on b_{\alpha} Y_{\beta}\ot X\t{}^{\alpha\beta}=X\on (bY)_{\alpha}\ot X\t{}^{\alpha}=\mu(X\ot b\,Y).
\end{align*}
\end{Rem}
If $\CL$ is an (anti)-left $B$-Hopf algebroid, we will denote $\lambda^{-1}|_{\CL\diamond_{B}1}:\CL\to \CL\ot_{\BB}\CL, X\mapsto X_{+}\ot X_{-}$, and $\mu^{-1}|_{1\diamond _{B}\CL}:\CL\to \CL\ot_{B}\CL, X\mapsto X_{[+]}\ot X_{[-]}$. Graphically, we will denote these two maps by
$$\gbeg 45
            \got 3{\CL}\gnl
            \gvac1\gcl1\gnl
            \glmpb\gnot{+-}\gcmpt\grmpb\gnl
            \gcl1\gvac1\gcl1\gnl
            \gob3{\CL\otimes_{\BB}\CL}
            \gend \qquad\textup{and}\qquad
            \gbeg 45
            \got 3{\CL}\gnl
            \gvac1\gcl1\gnl
            \glmpb\gnot{[+][-]}\gcmpt\grmpb\gnl
            \gcl1\gvac1\gcl1\gnl
            \gob3{\CL\otimes_{B}\CL}
            \gend.$$
The above two maps are called (anti)-translation maps, which determine $\lambda^{-1}$ and $\mu^{-1}$ in the sense that
$$\lambda^{-1}=\gbeg3{5}\gvac1\got3{\CL\diamond_{B}\CL}\gnl\gvac1\gcl1\gvac1\gcl1\gnl
 \glmpb\gnot{+-}\gcmpt\grmpb\gcl1\gnl\gcl1\gvac1\gmu\gnl
 \gob3{\,\CL\,\otimes_{\BB}\,\,\CL}
 \gend\qquad\textup{and}\qquad\mu^{-1}=\gbeg3{8}\gvac1\got1{\CL\diamond_{B}\CL}\gnl
 \gcl1\gvac1\gcl1\gnl\gcl1\glmpb\gnot{[+][-]}\gcmpt\grmpb\gnl
 \gibr\gvac1\gcl1\gnl
 \gcl1\gcn1113\gvac1\gcl1\gnl
 \gcl1\gvac1\gibr\gnl
 \gcl1\gvac1\gmu\gnl
 \gvac1\gob1{\,\CL\,\otimes_{B}\,\,\CL.}
 \gend
$$
Indeed, by the definition of the anti-translation maps we can check that $\mu^{-1}$ given above satisfies $\mu\circ\mu^{-1}=\id_{\CL\diamond_{B}\CL}$. Indeed,

$$\mu\circ\mu^{-1}=\gbeg4{12}\gvac1\got1{\CL\diamond_{B}\CL}\gnl
 \gcl1\gvac1\gcl1\gnl\gcl1\glmpb\gnot{[+][-]}\gcmpt\grmpb\gnl
 \gibr\gvac1\gcl1\gnl
 \gcl1\gcn1113\gvac1\gcl1\gnl
 \gcl1\gvac1\gibr\gnl
 \gcl1\gvac1\gmu\gnl
 \gcn1112\gvac1\gcn1121\gnl
 \gcmu\gcl1\gnl
\gcl1\gbr\gnl
\gmu\gcl1\gnl
\gob3{\CL\diamond_{B}\CL}
 \gend=\gbeg4{11}\gvac1\gvac1\got1{\CL\diamond_{B}\CL}\gnl
\gvac1\gibbr31\gnl
\glmpb\gnot{[+][-]}\gcmpt\grmpb\gcl1\gnl
\gcn1112\gvac1\gcl1\gcl1\gnl
\gcmu\gcl1\gcl1\gnl
\gcl1\gbr\gcl1\gnl
\gmu\gcl1\gcl1\gnl
\gcn1123\gvac1\gcl1\gcl1\gnl
\gvac1\gcl1\gbr\gnl
\gvac1\gmu\gcl1\gnl
\gvac1\gob4{\CL\diamond_{B}\CL}\gend=\gbeg4{5}\gvac1\got1{\CL\diamond_{B}\CL}\gnl
\gibbr31\gnl
\gcl1\gvac1\gcl1\gnl
\gbbr31\gnl
\gvac1\gob1{\CL\diamond_{B}\CL}\gend=\gbeg4{3}\gvac1\got1{\CL\diamond_{B}\CL}\gnl
\gcl1\gvac1\gcl1\gnl
\gvac1\gob1{\CL\diamond_{B}\CL}\gend.$$
As $\mu$ is bijective, it is the two-sides inverse of $\mu$. By the same method, we can also check $\lambda^{-1}$ can be determined by the translation map as above. 

Similarly to \cite[Prop.~3.7]{schau1}, we have
\begin{Prop}If $\CL$ is a left $B$-Hopf algebroid in $\B$, we have
            {\allowdisplaybreaks\begin{gather}
     \label{bydef1}
\gbeg 4{7}
        \got 3{\CL}\gnl
            \gvac1\gcl1\gnl
            \glmpb\gnot{+-}\gcmpt\grmpb\gnl
            \gcn1112\gvac1\gcl1\gnl
            \gcmu\gcl1\gnl
            \gcl1\gmu\gnl
            \gob2{\CL\diamond_{B}\CL}
            \gend=\gbeg 3{3}
        \got 3{\CL\diamond_{B}B}\gnl
            \gcl1\gvac1\gcl1\gnl
            \gob3{\CL\diamond_{B}\CL}
            \gend\qquad\textup{and}\qquad
            \gbeg 4{5}
       \gvac1\got3{\CL}\gnl
       \gvac1\gwcm3\gnl
        \glmpb\gnot{+-}\gcmpt\grmpb\gcl1\gnl
       \gcl1\gvac1\gmu\gnl
       \gvac1\gob1{\CL\,\,\ot_{\BB}\,\CL}
            \gend=\gbeg 3{3}
        \got 3{\CL\ot_{\BB}\BB}\gnl
            \gcl1\gvac1\gcl1\gnl
            \gob3{\CL\ot_{\BB}\CL}
            \gend
                      \\\notag\\
     \label{unittrans}
            \gbeg 45
            \got 3{B}\gnl
            \gvac1\gcl1\gnl
            \glmpb\gnot{+-}\gcmpt\grmpb\gnl
            \gcl1\gvac1\gcl1\gnl
            \gob3{\CL\otimes_{\BB}\CL}
            \gend=
            \gbeg 3{5}
        \got 1{B}\gvac1\gnl
        \gcl1\gvac1\gvac1\gnl
            \gcl1\gvac1\gu1\gnl
            \gcl1\gvac1\gcl1\gnl\gob3{\CL\ot_{\BB}\CL}\gend\qquad\textup{and}\qquad \gbeg 45
            \got 3{\BB}\gnl
            \gvac1\gcl1\gnl
            \glmpb\gnot{+-}\gcmpt\grmpb\gnl
            \gcl1\gvac1\gcl1\gnl
            \gob3{\CL\otimes_{\BB}\CL}
            \gend=
            \gbeg 3{5}
        \gvac1\got 3{\BB}\gnl
        \gvac1\gvac1\gcl1\gnl
            \gu1\gvac1\gmp{\ob}\gnl
            \gcl1\gvac1\gcl1\gnl
            \gob3{\CL\ot_{\BB}\CL}\gend
            \\\notag\\
     \label{transnabla}
            \gbeg 4{6}
            \got1{\BB}\gvac1\got1{\CL}\gnl
            \gcl1
            \gvac1\gcl1\gnl
            \gcl1\glmpb\gnot{+-}\gcmpt\grmpb\gnl
            \gcl1\gcl1\gvac1\gcl1\gnl
            \glm\gvac1\gcl1\gnl
            \gvac1\gvac1\gob1{\CL\ot_{\BB}\CL}
            \gend=\gbeg 4{6}
            \gvac1\got1{\BB}\gvac1\got1{\CL}\gnl
            \gvac1\gbbr31\gnl
            \glmpb\gnot{+-}\gcmpt\grmpb\gcl1\gnl
            \gcl1\gvac1\gcl1\gcl1\gnl
            \gcl1\gvac1\grm\gnl
            \gvac1\gob1{\CL\ot_{\BB}\CL}
            \gend\qquad\textup{and}\qquad\gbeg 35
            \got1{\CL}\gvac1\got1{\CL}\gnl
            \gwmu3\gnl
            \glmpb\gnot{+-}\gcmpt\grmpb\gnl
            \gcl1\gvac1\gcl1\gnl
            \gvac1\gob1{\CL\ot_{\BB}\CL}
            \gend
            =
            \gbeg 47
            \got1{\CL}\gvac2\got1{\CL}\gnl
            \gcl1\gvac2\gcl1\gnl
            \gdnot{+-}\glmptb\grmpb\gdnot{+-}\glmpb\grmptb\gnl
            \gcl1\gbr\gcl1\gnl
            \gmu\gbr\gnl
            \gcn2122\gmu\gnl
             \gvac1\gob2{\CL\ot_{\BB}\CL}
            \gend
             \\\notag\\
     \label{transleftcolinear}
            \gbeg 5{5}
           \gvac1\got2{\CL}\gnl
           \gwcm4\gnl
           \gcl1\gvac1\glmpb\gnot{+-}\gcmpt\grmpb\gnl
           \gcl1\gvac1\gcl1\gvac1\gcl1\gnl
           \gvac1\gvac1\gob1{\CL\diamond_{B}\CL\ot_{\BB}\CL}
            \gend
            =
            \gbeg 5{7}
             \gvac1\gvac1\got1{\CL}\gnl
             \gvac1\gvac1\gcl1\gnl
           \gvac1\glmpb\gnot{+-}\gcmpt\grmpb\gnl
           \gvac1\gcl1\gvac1\gcl1\gnl
           \gwcm3\gcl1\gnl
           \gcl1\gvac1\gcl1\gcn1113\gnl
           \gvac1\gvac1\gob1{\CL\diamond_{B}\CL\ot_{\BB}\CL}
            \gend\qquad\textup{and}\qquad \gbeg 5{7}
           \gvac1\gvac1\got1{\CL}\gnl
            \gvac1\gvac1\gcl1\gnl
           \gvac1\glmpb\gnot{+-}\gcmpt\grmpb\gnl
           \gvac1\gcl1\gvac1\gcl1\gvac1\gnl
           \glmpb\gnot{+-}\gcmpt\grmpb\gcl1\gnl
           \gcl1\gvac1\gbr\gnl
            \gend
            =
            \gbeg 4{7}
             \gvac1\gvac1\got1{\CL}\gnl
             \gvac1\gvac1\gcl1\gnl
           \gvac1\glmpb\gnot{+-}\gcmpt\grmpb\gnl
           \gvac1\gcl1\gvac1\gcl1\gnl
           \gvac1\gcl1\gwcm3\gnl
           \gvac1\gcl1\gcl1\gvac1\gcl1\gnl
            \gend
             \\\notag\\
     \label{transvanish}
            \gbeg 5{5}
           \gvac1\got1{\CL}\gnl
           \gvac1\gcl1\gnl
     \glmpb\gnot{+-}\gcmpt\grmpb\gnl
           \gcl1\gvac1\gcl1\gnl
           \gwmu 3\gnl
        \gvac1\got1{B}
            \gend
            =
            \gbeg 1{5}
             \got1{\CL}\gnl
             \gcl1\gnl
             \gmp{\varepsilon}\gnl
             \gcl1\gnl
             \gob1{B}
            \gend\qquad\textup{and}\qquad \gbeg 4{7}
           \gvac1\got1{\CL}\gnl
           \gvac1\gcl1\gnl
           \glmpb\gnot{+-}\gcmpt\grmpb\gnl
          \gcl1\gvac1\gcl1\gnl
          \gcl1\gvac1\gmp{\ol{\varepsilon}}\gnl
          \gwmu 3\gnl
          \gvac{1}\gob1{\CL}
            \gend
            =
            \gbeg 1{5}
            \got1{\CL}\gnl
            \gcl1\gnl
            \gcl1\gnl
            \gcl1\gnl
            \gob1{\CL}
            \gend
  \end{gather}}
  In (\ref{bydef1}), we identify $\cL$ with $\cL\diamond_{B}B$ and $\cL$ with $\cL\ot_{\BB}\BB$. If $\CL$ is an anti-left $B$-Hopf algebroid in $\B$, we have
            {\allowdisplaybreaks\begin{gather}
     \label{bydef2}
\gbeg 4{8}
        \got 3{\CL}\gnl
            \gvac1\gcl1\gnl
            \glmpb\gnot{[+][-]}\gcmpt\grmpb\gnl
            \gcn1112\gvac1\gcl1\gnl
            \gcmu\gcl1\gnl
            \gcl1\gbr\gnl
            \gmu\gcl1\gnl
            \gob3{\CL\diamond_{B}\CL}
            \gend=\gbeg 3{3}
        \got 3{\BB\diamond_{B}\CL}\gnl
            \gcl1\gvac1\gcl1\gnl
            \gob3{\CL\diamond_{B}\CL}
            \gend\qquad\textup{and}\qquad
          \gbeg4{9}\gvac1\got1{\CL}\gnl
          \gwcm3\gnl
 \gcl1\gvac1\gcl1\gnl\gcl1\glmpb\gnot{[+][-]}\gcmpt\grmpb\gnl
 \gibr\gvac1\gcl1\gnl
 \gcl1\gcn1113\gvac1\gcl1\gnl
 \gcl1\gvac1\gibr\gnl
 \gcl1\gvac1\gmu\gnl
 \gvac1\gob2{\,\CL\,\otimes_{B}\,\,\CL.}
 \gend=\gbeg 4{3}
        \got 3{\CL\ot_{B}B}\gnl
            \gcl1\gvac1\gcl1\gnl
            \gob3{\CL\ot_{B}\CL}
            \gend
                      \\\notag\\
     \label{unitantitrans}
            \gbeg 45
            \got 3{\BB}\gnl
            \gvac1\gcl1\gnl
            \glmpb\gnot{[+][-]}\gcmpt\grmpb\gnl
            \gcl1\gvac1\gcl1\gnl
            \gob3{\CL\otimes_{B}\CL}
            \gend=
            \gbeg 3{5}
        \got 1{\BB}\gvac1\gnl
        \gcl1\gvac1\gvac1\gnl
            \gcl1\gvac1\gu1\gnl
            \gcl1\gvac1\gcl1\gnl\gob3{\CL\ot_{B}\CL}\gend\qquad\textup{and}\qquad \gbeg 45
            \got 3{B}\gnl
            \gvac1\gcl1\gnl
            \glmpb\gnot{[+][-]}\gcmpt\grmpb\gnl
            \gcl1\gvac1\gcl1\gnl
            \gob3{\CL\otimes_{B}\CL}
            \gend=
            \gbeg 3{5}
        \gvac1\got 3{B}\gnl
        \gvac1\gvac1\gcl1\gnl
            \gu1\gvac1\gmp{\ob}\gnl
            \gcl1\gvac1\gcl1\gnl
            \gob3{\CL\ot_{B}\CL}\gend
            \\\notag\\
     \label{antitransnabla}
            \gbeg 4{6}
            \got1{B}\gvac1\got1{\CL}\gnl
            \gcl1
            \gvac1\gcl1\gnl
            \gcl1\glmpb\gnot{[+][-]}\gcmpt\grmpb\gnl
            \gcl1\gcl1\gvac1\gcl1\gnl
            \glm\gvac1\gcl1\gnl
            \gvac1\gvac1\gob1{\CL\ot_{B}\CL}
            \gend=\gbeg 4{6}
            \gvac1\got1{B}\gvac1\got1{\CL}\gnl
            \gvac1\gibbr31\gnl
            \glmpb\gnot{[+][-]}\gcmpt\grmpb\gcl1\gnl
            \gcl1\gvac1\gcl1\gcl1\gnl
            \gcl1\gvac1\grm\gnl
            \gvac1\gob1{\CL\ot_{B}\CL}
            \gend\qquad\textup{and}\qquad\gbeg 35
            \got1{\CL}\gvac1\got1{\CL}\gnl
            \gwmu3\gnl
            \glmpb\gnot{[+][-]}\gcmpt\grmpb\gnl
            \gcl1\gvac1\gcl1\gnl
            \gvac1\gob1{\CL\ot_{B}\CL}
            \gend
            =
            \gbeg 47
            \got1{\CL}\gvac2\got1{\CL}\gnl
            \gcl1\gvac2\gcl1\gnl
            \gdnot{[+][-]}\glmptb\grmpb\gdnot{[+][-]}\glmpb\grmptb\gnl
            \gcl1\gibr\gcl1\gnl
            \gmu\gibr\gnl
            \gcn2122\gmu\gnl
             \gvac1\gob2{\CL\ot_{B}\CL}
            \gend
             \\\notag\\
     \label{antitransleftcolinear}
            \gbeg 5{6}
           \gvac1\gvac1\got2{\CL}\gnl
           \gvac1\gwcm4\gnl
           \glmpb\gnot{[+][-]}\gcmpt\grmpb\gvac1\gcl1\gnl
           \gcl1\gvac1\gcl1\gvac1\gcl1\gnl
           \gcl1\gvac1\gcl1\gvac1\gcl1\gnl\gvac1\gvac1\gob1{(\CL\ot_{B}\CL)\diamond_{B}\CL}
            \gend
            =
            \gbeg 5{8}
             \gvac1\gvac1\got1{\CL}\gnl
             \gvac1\gvac1\gcl1\gnl
           \gvac1\glmpb\gnot{[+][-]}\gcmpt\grmpb\gnl
           \gvac1\gcl1\gvac1\gcl1\gnl
           \gwcm3\gcl1\gnl
           \gcl1\gvac1\gbr\gnl
           \gcl1\gvac1\gcl1\gcn2113\gnl
           \gvac1\gvac1\gob1{(\CL\ot_{B}\CL)\diamond_{B}\CL}
            \gend\qquad\textup{and}\qquad \gbeg 4{7}
           \gvac1\gvac1\got1{\CL}\gnl
            \gvac1\gvac1\gcl1\gnl
           \gvac1\glmpb\gnot{[+][-]}\gcmpt\grmpb\gnl
           \gvac1\gcl1\gvac1\gcl1\gvac1\gnl
           \glmpb\gnot{[+][-]}\gcmpt\grmpb\gcl1\gnl
           \gcl1\gvac1\gcl1\gcl1\gnl
            \gend
            =
            \gbeg 4{7}
             \gvac1\gvac1\got1{\CL}\gnl
             \gvac1\gvac1\gcl1\gnl
           \gvac1\glmpb\gnot{[+][-]}\gcmpt\grmpb\gnl
           \gvac1\gcl1\gvac1\gcl1\gnl
           \gvac1\gcl1\gwcm3\gnl
           \gvac1\gcl1\gcl1\gvac1\gcl1\gnl
            \gend
             \\\notag\\
     \label{antitransvanish}
            \gbeg 5{5}
           \gvac1\got1{\CL}\gnl
           \gvac1\gcl1\gnl
     \glmpb\gnot{[+][-]}\gcmpt\grmpb\gnl
           \gcl1\gvac1\gcl1\gnl
           \gwmu 3\gnl
        \gvac1\got1{B}
            \gend
            =
            \gbeg 1{5}
             \got1{\CL}\gnl
             \gcl1\gnl
             \gmp{\ol{\varepsilon}}\gnl
             \gcl1\gnl
             \gob1{\BB}
            \gend\qquad\textup{and}\qquad \gbeg 4{7}
           \gvac1\got1{\CL}\gnl
           \gvac1\gcl1\gnl
           \glmpb\gnot{[+][-]}\gcmpt\grmpb\gnl
          \gcl1\gvac1\gcl1\gnl
          \gcl1\gvac1\gmp{\varepsilon}\gnl
          \gwmu 3\gnl
          \gvac{1}\gob1{\CL}
            \gend
            =
            \gbeg 1{5}
            \got1{\CL}\gnl
            \gcl1\gnl
            \gcl1\gnl
            \gcl1\gnl
            \gob1{\CL}
            \gend
  \end{gather}}
  In (\ref{bydef2}), we identify $\cL$ with $\BB\diamond_{B}\CL$ and $\cL$ with $\cL\ot_{B}B$. If $\CL$ is a $B$-Hopf algebroid, then we have
              {\allowdisplaybreaks\begin{gather}
     \label{hopf1}
           \gbeg 5{8}
           \gvac1\gvac1\got1{\CL}\gnl
            \gvac1\gvac1\gcl1\gnl
           \gvac1\glmpb\gnot{[+][-]}\gcmpt\grmpb\gnl
           \gvac1\gcl1\gvac1\gcl1\gvac1\gnl
           \gvac1\gcl1\glmpb\gnot{+-}\gcmpt\grmpb\gnl
           \gvac1\gcl1\gcl1\gvac1\gcl1\gnl
            \gend=\gbeg 4{6}
            \gvac1\gvac1\got1{\CL}\gnl
            \gvac1\gwcm3\gnl
            \gvac1\gibbr31\gnl
            \glmpb\gnot{[+][-]}\gcmpt\grmpb\gcl1\gnl
            \gcl1\gvac1\gcl1\gcl1\gnl
            \gend
             \\\notag\\
     \label{hopf2}
            \gbeg 5{8}
           \gvac1\gvac1\got1{\CL}\gnl
            \gvac1\gvac1\gcl1\gnl
           \gvac1\glmpb\gnot{+-}\gcmpt\grmpb\gnl
           \gvac1\gcl1\gvac1\gcl1\gvac1\gnl
           \gvac1\gcl1\glmpb\gnot{[+][-]}\gcmpt\grmpb\gnl
           \gvac1\gcl1\gcl1\gvac1\gcl1\gnl
            \gend=\gbeg 4{6}
            \gvac1\gvac1\got1{\CL}\gnl
            \gvac1\gwcm3\gnl
            \glmpb\gnot{+-}\gcmpt\grmpb\gcl1\gnl
            \gcl1\gvac1\gcl1\gcl1\gnl
            \gend
             \\\notag\\
     \label{hopf3}
           \gbeg 5{7}
           \gvac1\gvac1\got1{\CL}\gnl
            \gvac1\gvac1\gcl1\gnl
           \gvac1\glmpb\gnot{[+][-]}\gcmpt\grmpb\gnl
           \gvac1\gcl1\gvac1\gcl1\gvac1\gnl
           \glmpb\gnot{+-}\gcmpt\grmpb\gcl1\gnl
           \gcl1\gvac1\gbr\gnl
            \gend= \gbeg 4{7}
           \gvac1\gvac1\got1{\CL}\gnl
            \gvac1\gvac1\gcl1\gnl
           \gvac1\glmpb\gnot{+-}\gcmpt\grmpb\gnl
           \gvac1\gcl1\gvac1\gcl1\gvac1\gnl
           \glmpb\gnot{[+][-]}\gcmpt\grmpb\gcl1\gnl
           \gcl1\gvac1\gcl1\gcl1\gnl
            \gend
  \end{gather}}
\end{Prop}  
\begin{proof}
We only show the non-trivial part. The proof for the second equation of (\ref{transnabla}) is similar to (\ref{antitransnabla}) which will be shown later. For the second equation of (\ref{transleftcolinear}), we observe that the target of the map is $\int_{c,d}{}_{}\cL_{\Bar{c}}\ot {}_{\Bar{d}}\cL_{}\ot {}_{\Bar{c},d}\cL$. Define map $\lambda_{1,3}:\int_{c,d}{}_{}\cL_{\Bar{c}}\ot {}_{\Bar{d}}\cL_{}\ot {}_{\Bar{c},d}\cL\to \int_{c,d}{}_{\Bar{c}}\cL_{\Bar{d}}\ot {}_{\Bar{d}}\cL_{}\ot {}_{c}\cL$ 
 by 
 $$\lambda_{1,3}:=\gbeg5{6}\got5{\int_{c,d}{}_{}\cL_{\Bar{c}}\ot {}_{\Bar{d}}\cL_{}\ot {}_{\Bar{c},d}\cL}\gnl
 \gcmu\gvac1\gcl1\gvac1\gcl1\gnl
 \gcl1\gbbr31\gvac1\gcl1\gnl
 \gcl1\gcl1\gvac1\gwmu3\gnl
 \gcl1\gcn1113\gvac1\gvac1\gcl1\gnl
 \gob3{\int_{c,d}{}_{\Bar{c}}\cL_{\Bar{d}}\ot {}_{\Bar{d}}\cL_{}\ot {}_{c}\cL}
 \gend
 $$
 We can check that $\lambda_{1,3}$ factors through all the balanced tensor products. Also, we can define a map $\lambda\ot \id:\int_{c,d}{}_{\Bar{c}}\cL_{\Bar{d}}\ot {}_{\Bar{d}}\cL_{}\ot {}_{c}\cL\to \cL\diamond_{B}\cL\diamond_{B}\cL$. Apply $(\lambda\ot \id)\circ\lambda_{1,3}$ on the left hand side of the second equation of (\ref{transleftcolinear}), we get
 $$\gbeg 5{10}
           \gvac1\gvac1\gvac1\got1{\CL}\gnl
            \gvac1\gvac1\gvac1\gcl1\gnl
           \gvac1\gvac1\glmpb\gnot{+-}\gcmpt\grmpb\gnl
           \gvac1\gvac1\gcl1\gvac1\gcl1\gvac1\gnl
           \gvac1\glmpb\gnot{+-}\gcmpt\grmpb\gcl1\gnl
           \gvac1\gcl1\gvac1\gbr\gnl
           \gwcm3\gcl1\gcl1\gnl
           \gcn1112\gvac1\gbr\gcl1\gnl
           \gcmu\gcl1\gmu\gnl
           \gcl1\gmu\gcn2122\gnl
            \gend=\gbeg 4{7}
        \got 3{\CL}\gnl
            \gvac1\gcl1\gnl
            \glmpb\gnot{+-}\gcmpt\grmpb\gnl
            \gcn1112\gvac1\gcl1\gnl
            \gcmu\gcl1\gnl
            \gcl1\gmu\gnl
            \gob2{\CL\diamond_{B}\CL}
            \gend=\gbeg 3{3}
        \got 3{\CL\diamond_{B}B}\gnl
            \gcl1\gvac1\gcl1\gnl
            \gob3{\CL\diamond_{B}\CL}
            \gend.$$
  Applying $(\lambda\ot \id)\circ\lambda_{1,3}$ on the right hand side of the second equation of (\ref{transleftcolinear}), we get  
$$\gbeg 5{10}
           \gvac1\gvac1\got1{\CL}\gnl
             \gvac1\gvac1\gcl1\gnl
           \gvac1\glmpb\gnot{+-}\gcmpt\grmpb\gnl
           \gvac1\gcl1\gvac1\gcl1\gnl
           \gvac1\gcl1\gwcm3\gnl
           \gvac1\gcl1\gcn2113\gcl1\gnl
           \gwcm3\gcl1\gcl1\gnl
           \gcn1112\gvac1\gbr\gcl1\gnl
           \gcmu\gcl1\gmu\gnl
           \gcl1\gmu\gcn2122\gnl
            \gend=\gbeg 5{7}
           \gvac1\gvac1\got1{\CL}\gnl
             \gvac1\gvac1\gcl1\gnl
           \gvac1\glmpb\gnot{+-}\gcmpt\grmpb\gnl
           \gvac1\gcl1\gvac1\gcl1\gnl
           \gwcm3\gcl1\gnl
           \gcl1\gvac1\gmu\gnl
           \gcl1\gvac1\gcn2122\gnl
           \gcl1\gvac1\gcmu\gnl
            \gend=\gbeg 3{3}
        \got 3{\CL\diamond_{B}B}\gnl
            \gcl1\gvac1\gcl1\gnl
            \gob3{\CL\diamond_{B}\CL}
            \gend. $$
            For the 2nd equation of  (\ref{antitransnabla}), applying $\mu$ on the right hand side, we have
             \begin{multline*} \gbeg 4{10}
            \got1{\CL}\gvac2\got1{\CL}\gnl
            \gcl1\gvac2\gcl1\gnl
            \gdnot{[+][-]}\glmptb\grmpb\gdnot{[+][-]}\glmpb\grmptb\gnl
            \gcl1\gibr\gcl1\gnl
            \gmu\gibr\gnl
            \gcn2122\gmu\gnl
            \gcmu\gcn2121\gnl
\gcl1\gbr\gnl
\gmu\gcl1\gnl
\gob3{\CL\diamond_{B}\CL}
            \gend=\gbeg 5{10}
            \gvac1\got1{\CL}\gvac1\gvac1\got1{\CL}\gnl
            \gvac1\gcl1\gvac1\gvac1\gcl1\gnl
            \glmpb\gnot{[+][-]}\gcmpt\grmpb\gvac1\gcl1\gnl
            \gcl1\gvac1\gibbr31\gnl
            \gcl1\glmpb\gnot{[+][-]}\gcmpt\grmpb\gcl1\gnl
            \gmu\gvac1\gmu\gnl
            \gwcm2\gvac1\gcn3121\gnl
            \gcl1\gbbr31\gnl
            \gmu\gvac1\gcl1\gnl
            \gvac1\gob2{\CL\,\diamond_{B}\,\,\CL}
            \gend=\gbeg 5{14}
            \gvac1\got1{\CL}\gvac1\gvac1\got1{\CL}\gnl
            \gvac1\gcl1\gvac1\gvac1\gcl1\gnl
            \glmpb\gnot{[+][-]}\gcmpt\grmpb\gvac1\gcl1\gnl
            \gcl1\gvac1\gibbr31\gnl
            \gcl1\glmpb\gnot{[+][-]}\gcmpt\grmpb\gcl1\gnl
            \gcl1\gcn2113\gmu\gnl
            \gcn2112\gcn1112\gcn2123\gnl
            \gcmu\gcmu\gcl1\gnl
\gcl1\gbr\gcl1\gcl1\gnl
\gmu\gmu\gcl1\gnl
\gcn2122\gcn1123\gvac1\gcl1\gnl
\gcn2122\gvac1\gbr\gnl
\gwmuh427\gcl1\gnl
           \gvac1\gvac1\gob2{\CL\,\,\diamond_{B}\,\,\CL}
            \gend=\gbeg 5{14}
            \gvac1\got1{\CL}\gvac1\gvac1\got1{\CL}\gnl
            \gvac1\gcl1\gvac1\gvac1\gcl1\gnl
            \glmpb\gnot{[+][-]}\gcmpt\grmpb\gvac1\gcl1\gnl
            \gcl1\gvac1\gibbr31\gnl
            \gcl1\glmpb\gnot{[+][-]}\gcmpt\grmpb\gcl1\gnl
            \gcl1\gcn2113\gmu\gnl
            \gcn2112\gcn1112\gcn2123\gnl
            \gcmu\gcmu\gcl1\gnl
             \gcl1\gcl1\gcl1\gbr\gnl
             \gcl1\gcl1\gmu\gcl1\gnl
             \gcl1\gcl1\gcn1121\gvac1\gcl1\gnl
             \gcl1\gbr\gvac1\gcl1\gnl
             \gmu\gwmuh415\gnl
          \gvac1\gob3{\CL\,\,\diamond_{B}\,\,\CL}
            \gend\\
            \gbeg 5{15}
            \gvac1\got1{\CL}\gvac1\gvac1\got1{\CL}\gnl
            \gvac1\gcl1\gvac1\gvac1\gcl1\gnl
            \glmpb\gnot{[+][-]}\gcmpt\grmpb\gvac1\gcl1\gnl
            \gcl1\gvac1\gibbr31\gnl
            \gcl1\glmpb\gnot{[+][-]}\gcmpt\grmpb\gcl1\gnl
            \gcl1\gcn1112\gvac1\gcl1\gcl1\gnl
            \gcl1\gcmu\gcl1\gcl1\gnl
            \gcl1\gcl1\gbr\gcl1\gnl
            \gcl1\gmu\gbr\gnl
            \gcn1112\gcn2123\gcl1\gcl1\gnl
            \gcmu\gmu\gcl1\gnl
            \gcl1\gcn1111\gcn2121\gcl1\gnl
            \gcn1111\gbr\gvac1\gcl1\gnl
            \gmu\gwmu3\gnl
          \gvac1\gob2{\CL\,\,\diamond_{B}\,\,\CL}
            \gend=\gbeg 5{8}
        \got 3{\CL}\got1{\CL}\gnl
            \gvac1\gcl1\gvac1\gcl5\gnl
            \glmpb\gnot{[+][-]}\gcmpt\grmpb\gnl
            \gcn1112\gvac1\gcl1\gnl
            \gcmu\gcl1\gnl
            \gcl1\gbr\gnl
            \gmu\gmu\gnl
            \gvac1\gob2{\CL\diamond_{B}\CL}
            \gend= \gbeg5{4} 
  \gvac1\got2{\BB\diamond_{B}\CL\ot\CL}\gnl
  \gcl1\gvac1\gmu\gnl
  \gcl1\gvac1\gcn2122\gnl
 \gob3{\CL\,\,\diamond_{B}\,\,\,\CL}
\gend 
            \end{multline*}
which is equal to applying $\mu$ on the left hand side. For the first equation of (\ref{antitransleftcolinear}), we note that the target of the map is $\int_{a,b}{}_{\ol a}\cL_{b}\ot {}_{b}\cL\ot {}_{a}\cL$. Apply $\mu\diamond_{B}\id$ on the right hand side, we have       
$$\gbeg 5{11}
             \gvac1\gvac1\gvac1\got1{\CL}\gnl
             \gvac1\gvac1\gvac1\gcl1\gnl
           \gvac1\gvac1\glmpb\gnot{[+][-]}\gcmpt\grmpb\gnl
           \gvac1\gvac1\gcl1\gvac1\gcl1\gnl
           \gvac1\gwcm3\gcl1\gnl
           \gvac1\gcl1\gvac1\gbr\gnl
           \gvac1\gcl1\gvac1\gcl1\gcn2113\gnl
           \gwcm3\gcl1\gvac1\gcl1\gnl
           \gcl1\gvac1\gbr\gvac1\gcl1\gnl
           \gwmu3\gcl1\gvac1\gcl1\gnl
           \gvac1\gvac1\gvac1\gob1{\cL\diamond_{B}\cL\diamond_{B}\CL}
            \gend=\gbeg 4{9}
        \got 3{\CL}\gnl
            \gvac1\gcl1\gnl
            \glmpb\gnot{[+][-]}\gcmpt\grmpb\gnl
            \gcn1112\gvac1\gcl1\gnl
            \gcmu\gcl1\gnl
            \gcl1\gbr\gnl
            \gmu\gcl1\gnl
            \gcn1122\gwcm3\gnl
            \gvac1\gvac1\gob1{\cL\diamond_{B}\cL\diamond_{B}\CL}
            \gend=\gbeg4{4} \gvac1\got1{\BB\diamond_{B}\CL}\gnl
          \gcl2\gwcm3\gnl
          \gvac1\gcl1\gvac1\gcl1\gnl
          \gvac1\gvac1\gob1{\cL\diamond_{B}\cL\diamond_{B}\CL}
 \gend,$$
 which is equal to $\mu\diamond_{B}\id$ applying on the left hand side. For (\ref{antitransvanish}), we have
 $$\gbeg 4{7}
           \gvac1\got1{\CL}\gnl
           \gvac1\gcl1\gnl
           \glmpb\gnot{[+][-]}\gcmpt\grmpb\gnl
          \gcl1\gvac1\gcl1\gnl
          \gcl1\gvac1\gmp{\varepsilon}\gnl
          \gwmu 3\gnl
          \gvac{1}\gob1{\CL}
            \gend=\gbeg 4{9}
           \gvac1\got1{\CL}\gnl
           \gvac1\gcl1\gnl
           \glmpb\gnot{[+][-]}\gcmpt\grmpb\gnl
           \gcn2112\gcl1\gnl
          \gcmu\gcl1\gnl
          \gmp{\varepsilon}\gcl1\gmp{\varepsilon}\gnl
          \gmu\gcl1\gnl
          \gwmuh425 \gnl
          \gvac{1}\gob2{\CL}
            \gend
            =\gbeg 4{8}
           \gvac1\got1{\CL}\gnl
           \gvac1\gcl1\gnl
           \glmpb\gnot{[+][-]}\gcmpt\grmpb\gnl
           \gcn2112\gcl1\gnl
          \gcmu\gmp{\varepsilon}\gnl
          \gmp{\varepsilon}\grm\gnl
          \gmu\gnl
          \gob2{\CL}
            \gend=\gbeg 4{11}
           \gvac1\got1{\CL}\gnl
           \gvac1\gcl1\gnl
           \glmpb\gnot{[+][-]}\gcmpt\grmpb\gnl
           \gcn2112\gcl1\gnl
          \gcmu\gmp{\ol \varepsilon}\gnl
          \gcl1\gbr\gnl
          \gmu\gcl1\gnl
          \gcn2121\gcl1\gnl
          \gmp{\varepsilon}\gvac1\gcl1\gnl
          \gwmu3\gnl
          \gob3{\CL}
            \gend=\gbeg 4{11}
           \gvac1\got1{\CL}\gnl
           \gvac1\gcl1\gnl
           \glmpb\gnot{[+][-]}\gcmpt\grmpb\gnl
           \gcn2112\gcl1\gnl
          \gcmu\gcl1\gnl
          \gcl1\gbr\gnl
          \gmu\gcl1\gnl
          \gcn2121\gcl1\gnl
          \gmp{\varepsilon}\gvac1\gcl1\gnl
          \gwmu3\gnl
          \gob3{\CL}
            \gend=\id_{\CL}$$.
 For (\ref{hopf1}), apply $\id\ot \lambda$ on the right hand side, we have
 $$\gbeg 4{7}
            \gvac1\gvac1\got1{\CL}\gnl
            \gvac1\gwcm3\gnl
            \gvac1\gibbr31\gnl
            \glmpb\gnot{[+][-]}\gcmpt\grmpb\gcl1\gnl
            \gcl1\gcn2132\gcl1\gnl
            \gcl1\gcmu\gcl1\gnl
            \gcl1\gcl1\gmu\gnl
            \gend=\gbeg 5{7}
            \gvac1\gvac1\gvac1\got1{\CL}\gnl
            \gvac1\gvac1\gwcm3\gnl
            \gvac1\gvac1\gibbr31\gnl
           \gvac1\glmpb\gnot{[+][-]}\gcmpt\grmpb\gcl2\gnl
           \gvac1\gcl1\gvac1\gcl1\gvac1\gnl
           \glmpb\gnot{[+][-]}\gcmpt\grmpb\gmu\gnl
           \gcl1\gvac1\gcl1\gcn2122\gnl
            \gend=\gbeg 33\got3{\CL}\gnl
            \gvac1\gcl1\gnl
            \glmpb\gnot{[+][-]}\gcmpt\grmpb\gnl
            \gcl1\gvac1\gcl1\gnl
            \gend,$$
            which is equal to $\id\ot \lambda$ applying on the left hand side. We can prove (\ref{hopf2})
 by comparing the result of $\id_{\cL}\ot \mu$ on both sides. Here the first step use (\ref{antitransleftcolinear}). For (\ref{hopf3}), apply $\mu\ot\id$ on the left hand side, we have
 $$  \gbeg 5{9}
           \gvac1\gvac1\gvac1\got1{\CL}\gnl
            \gvac1\gvac1\gvac1\gcl1\gnl
           \gvac1\gvac1\glmpb\gnot{[+][-]}\gcmpt\grmpb\gnl
           \gvac1\gvac1\gcl1\gvac1\gcl1\gvac1\gnl
           \gvac1\glmpb\gnot{+-}\gcmpt\grmpb\gcl1\gnl
           \gvac1\gcl1\gvac1\gbr\gnl
           \gwcm3\gcl1\gcl1\gnl
           \gcl1\gvac1\gbr\gcl1\gnl
           \gwmu3\gcl1\gcl1\gnl
            \gend=\gbeg 5{9}
           \gvac1\gvac1\got1{\CL}\gnl
            \gvac1\gvac1\gcl1\gnl
           \gvac1\glmpb\gnot{[+][-]}\gcmpt\grmpb\gnl
           \gvac1\gcl1\gvac1\gcl1\gvac1\gnl
           \gwcm3\gcn1113\gnl
            \gcl1\glmpb\gnot{+-}\gcmpt\grmpb\gcl1\gnl
            \gcl1\gcl1\gvac1\gbr\gnl
            \gcl1\gbbr31\gcl1\gnl
            \gmu\gvac1\gcl1\gcl1\gnl
            \gend=\gbeg 5{9}
           \gvac1\gvac1\got1{\CL}\gnl
            \gvac1\gvac1\gcl1\gnl
           \gvac1\glmpb\gnot{[+][-]}\gcmpt\grmpb\gnl
           \gvac1\gcl1\gvac1\gcl1\gvac1\gnl
           \gwcm3\gcn1111\gnl
           \gcl1\gvac1\gbr\gnl
           \gcl1\gvac1\gcl1\gcn1113\gnl
           \gwmu3\glmpb\gnot{+-}\gcmpt\grmpb\gnl
           \gvac1\gcl1\gvac1\gcl1\gvac1\gcl1\gnl
            \gend=\gbeg 33\got3{\CL}\gnl
            \gvac1\gcl1\gnl
            \glmpb\gnot{+-}\gcmpt\grmpb\gnl
            \gcl1\gvac1\gcl1\gnl
            \gend,$$
            which is equal to $\mu\ot\id$ applying on the right hand side. Here the first step uses (\ref{transleftcolinear}).
\end{proof}
\begin{Def}
   Given a $B$-Hopf algebroid $\cL$ in $\B$, a Hopf ideal $\cI$ of $\cL$ is  a $B^{e}$-subbimodule of $\cL$ in $\B$ such that
\begin{itemize}
    \item [(i)] $\cI$ is an ideal of $\cL$.
    \item[(ii)] $\cI$ is a coideal of $\cL$, i.e. $\Delta(\cI)\subseteq \cI\diamond_{B}\cL+\cL\diamond_{B}\cI$ and $\varepsilon(\cI)=0$.
    \item[(iii)] For all $j\in\cI$,
    \[j_{+}\ot_{\overline{B}}j_{-}\in \cI\ot_{\overline{B}}\cL+\cL\ot_{\overline{B}}\cI. \]
     \item[(iv)] For all $j\in\cI$,
    \[j_{[+]}\ot_{B}j_{[-]}\in \cI\ot_{B}\cL+\cL\ot_{B}\cI. \]
\end{itemize}
\end{Def}
It is similar to \cite{Gho}, we have
\begin{Prop}\label{prop. quotient Hopf algebroids}
    If $\cL$ is a $B$-Hopf algebroid in $\B$ and $\cI$ is a Hopf ideal of $\cL$ then $\cL/\cI$ is a $B$-Hopf algebroid in $\B$.
\end{Prop}

\begin{Expl}
For a braided Hopf algebra $H$ in $\B$ \cite{Majid1994,schau2,schau3}, we will use
$$S=\gbeg15\got1H\gnl\gcl1\gnl\gmp{S}\gnl\gcl1\gnl\gob1H\gend\qquad\text{ and
}\qquad
S\inv=\gbeg15\got1H\gnl\gcl1\gnl\gmp{\hat{S}}\gnl\gcl1\gnl\gob1H\gend$$
for the antipode  and its inverse in case $S$ is bijective. We note that the antipode is an
anti-coalgebra map and anti-algebra map in the sense that
\begin{align*}
  \gbeg 34
  \got 3H\gnl
  \gvac 1\gmp S\gnl
  \gwcm 3\gnl
  \gob 1H\gvac 1\gob 1H
  \gend
  &=
  \gbeg 25
  \got 2H\gnl
  \gcmu\gnl
  \gbr\gnl
  \gmp S\gmp S\gnl
  \gob 1H\gob 1H
  \gend
&\text{and}&&
  \gbeg 34
  \got 1H\gvac 1\got 1H\gnl
  \gwmu 3\gnl
  \gvac 1\gmp S\gnl
  \gob 3H
  \gend
  &=
  \gbeg 25
  \got 1H\got 1H\gnl
  \gmp S\gmp S\gnl
  \gbr\gnl
  \gmu\gnl
  \gob 2H
  \gend.
\end{align*}
If the antipode is bijective, we have
$$\gbeg2{6}\got2{H}\gnl
\gcmu\gnl
\gibr\gnl
\gmp{\hat{S}}\gcl1\gnl
\gmu\gnl
\gob2{H}\gend=\gbeg1{4}\got1H\gnl\gcu1\gnl
\gu1\gnl\gob1H\gend=\gbeg2{6}\got2{H}\gnl
\gcmu\gnl
\gibr\gnl
\gcl1\gmp{\hat{S}}\gnl
\gmu\gnl
\gob2{H}\gend.$$
For $H$ with bijective antipode, we have
$$\gbeg 45
            \got 3{H}\gnl
            \gvac1\gcl1\gnl
            \glmpb\gnot{+-}\gcmpt\grmpb\gnl
            \gcl1\gvac1\gcl1\gnl
            \gob1H\gvac{1}\gob1H\gnl
            \gend=\gbeg2{4}\got2{H}\gnl
            \gcmu\gnl
            \gcl1\gmp{S}\gnl
            \gob1{H}\gob1{H}\gnl
            \gend
            \qquad\textup{and}\qquad
            \gbeg 45
            \got 3{H}\gnl
            \gvac1\gcl1\gnl
            \glmpb\gnot{[+][-]}\gcmpt\grmpb\gnl
            \gcl1\gvac1\gcl1\gnl
             \gob1H\gvac{1}\gob1H
            \gend=\gbeg2{5}\got2{H}\gnl
            \gcmu\gnl
            \gmp{\hat{S}}\gcl1\gnl
            \gibr\gnl
            \gob1{H}\gob1{H}\gnl
            \gend.$$
\end{Expl}
% \rosso{In particular, every left $H$-comodule $V$ is a right
% $H$-comodule, and, if $S$ is an isomorphism every left comodule
% $W$ is a right comodule with comodule structures
% \begin{equation}\label{switchover}
%   \gbeg25
%   \gvac1\got1V\gnl
%   \glcm\gnl
%   \gbr\gnl
%   \gcl1\gmp+\gnl
%   \gob1V\gob1H\gend
%   \qquad\text{ resp. }
%   \qquad
%   \gbeg25
%   \got1W\gnl
%   \grcm\gnl
%   \gibr\gnl
%   \gmp-\gcl1\gnl
%   \gob1H\gob1W
%   \gend
% \end{equation}

% A Yetter-Drinfeld module over a Hopf algebra $H$ in $\B$, as
% defined in the braided case by Bespalov \cite{Bes:CMQGBC} is a
% right $H$-comodule and right $H$-module $V$ satisfying
% \[\gbeg38
%   \got1V\got2H\gnl
%   \gcl1\gcmu\gnl
%   \gbr\gcl1\gnl
%   \gcl2\grm\gnl
%   \gvac1\grcm\gnl
%   \gbr\gcl1\gnl
%   \gcl1\gmu\gnl
%   \gob1V\gob2H
%   \gend
%   =
%   \gbeg45
%   \got1V\gvac1\got2H\gnl
%   \grcm\gcmu\gnl
%   \gcl1\gbr\gcl1\gnl
%   \grm\gmu\gnl
%   \gob1V\gvac1\gob2H
%   \gend \]
% In \cite{Bes:CMQGBC} it is shown in particular that if the
% antipode of $H$ is an isomorphism, then the category of
% Yetter-Drinfeld modules is braided, with the braiding and its
% inverse defined by
% \begin{equation*}
%   \sigma_{VW}=
%   \gbeg35
%   \got1V\got1W\gnl
%   \gcl1\grcm\gnl
%   \gbr\gcl1\gnl
%   \gcl1\grm\gnl
%   \gob1W\gob1V
%   \gend
%   \qquad
%   \text{ and }
%   \qquad
%   \sigma_{VW}\inv=
%   \gbeg37
%   \got1W\gvac1\got1V\gnl
%   \grcm\gcl1\gnl
%   \gcl3\gibr\gnl
%   \gvac1\gcl1\gmp-\gnl
%   \gvac1\grm\gnl
%   \gibr\gnl
%   \gob1V\gob1{W.}
%   \gend
% \end{equation*}
% In particular, $\sigma^{\pm 1}$ are $H$-colinear isomorphisms.}

\begin{Def}
Let $\CL$ be a left $B$-bialgebroid in $\B$. A right $\CL$-comodule is a $\BB$-bimodule $P$ together with a  map $\delta:P\to P\times_{B}\CL$ denoted by
$$\gbeg2{4}\got1P\gnl
\grcm\gnl
\gcl1\gcl1\gnl
\gob1{P\times_{B}\CL}
\gend$$
such that
\begin{itemize}
\item[(i)] $\delta$ is $\BB$-bilinear:
$$\gbeg4{6}\got1{\BB}\got1P\got1{\BB}\gnl
\glm\gcl1\gnl
\gvac1\grm\gnl
\gvac1\grcm\gnl
\gvac1\gcl1\gcl1\gnl
\gvac1\gob1{P\times_{B}\CL}
\gend=\gbeg4{7}\got1{\BB}\got1P\got1{\BB}\gnl
\gcl1\gcl1\gcn2113\gnl
\gcl1\grcm\gcl1\gnl
\gbr\gcl1\gcl1\gnl
\gcl1\glm\gcl1\gnl
\gcl1\gvac1\grm\gnl
\gvac1\gob1{P\times_{B}\CL}
\gend$$
\item[(ii)] $\alpha\circ (\delta\times_{B}\id_{\cL})\circ \delta=\alpha'\circ (\id_{P}\times_{B} \Delta)\circ \delta:P\to P\times_{B} \cL \times_{B} \cL$. Graphically,
$$\gbeg4{7}\got1P\gnl
\gcl1\gnl
\grcm\gnl
\gcl1\gcn2113\gnl
\grcm\gcl1\gnl
\gcl1\gcl1\gcl1\gnl
\gvac1\gob1{P\times_{B}\CL\times \CL}
\gend=\gbeg4{7}\got1P\gnl
\gcl1\gnl
\grcm\gnl
\gcl1\gcn2112\gnl
\gcl1\gcmu\gnl
\gcl1\gcl1\gcl1\gnl
\gvac1\gob1{P\times_{B}\CL\times \CL}
\gend$$
\item[(iii)] $(\id_{P}\diamond_{B} \varepsilon)\circ \delta=\id_{P}$. Graphically,
$$\gbeg4{6}\got1P\gnl
\grcm\gnl
\gcl1\gmp{\ol \varepsilon}\gnl
\gibr\gnl
\glm\gnl
\gvac1\gob1{P}
\gend=\gbeg4{5}\got1P\gnl
\gcl1\gnl
\gcl1\gnl
\gcl1\gnl
\gob1{P}
\gend$$
\end{itemize}
In addition, if $P$ is a $\BB$-ring such that
$$\gbeg3{5} 
  \got 1{P}\gvac 1\got 1{P}\gnl
  \gwmu 3\gnl
  \gvac1\grcm3\gnl
\gvac1\gcl1\gcl1\gnl
\gvac1\gob2{P\times_{B}\cL}
\gend
=\gbeg4{5}
 \got1{P}\gvac1\got1{P}\gnl
 \grcm\grcm\gnl
\gcl1\gbr\gcl1\gnl
\gmu\gmu\gnl
\gvac1\gob2{P\times_{B}\cL}
\gend$$
we call $P$ a right $\cL$-comodule algebra.
% If $P$ and $Q$ are two right $\CL$-comodule, a map $F:P\to Q$ is called right $\CL$-colinear if $F$ is $\BB$-\rosso{bilinear}   such that $\delta\circ F=(F\diamond_{B}\id)\circ\delta:P\to Q\diamond_{B}\CL$.
\end{Def}
Given a left $B$-bialgebroid $\CL$.
If $V$ and $W$ are two right $\CL$-comodules, then $V\ot_{\BB} W$ is a right $\CL$-comodule with the codiagonal coaction given by
\[\gbeg45
\gvac1\got1{V\ot_{\BB}W}\gnl
  \grcm\grcm\gnl
  \gcl2\gbr\gcl1\gnl
  \gvac1\gcl1\gmu\gnl\gvac1\gob1{(V\ot_{\BB}W)\diamond_{B}\CL}
  \gend\]
  We denote the category of right $\CL$-comodules by $(\B^{\cL},\ot_{\BB})$. Similarly, a left $\cL$-comodule $P$ is a $B$-bimodule together with a $B$-bilinear map $\delta:P\to \cL\times P,\quad p\mapsto p\mo\ot p\z$ in the sense that
  $$\delta:=\gbeg2{4}\gvac{1}\got1P\gnl
  \glcm\gnl
  \gcl1\gcl1\gnl
  \gvac1\gob1{\cL\times_{B}P}
  \gend\qquad\textup{such that}\qquad\gbeg4{6}\got1B\got1P\got1B\gnl
  \glm\gcl1\gnl
  \gvac1\grm\gnl
  \glcm\gnl
  \gcl1\gcl1\gnl
  \gvac1\gob1{\cL\times_{B}P}
  \gend=\gbeg4{6}\got1B\gvac1\got1P\got1B\gnl
\gcl1\glcm\gcl1\gnl
\glm\gbr\gnl
\gvac1\grm\gcl1\gnl
\gvac1\gcl1\gvac1\gcl1\gnl
 \gvac1 \gvac1\gob1{\cL\times_{B}P}
  \gend$$
which is coassociative and counital in the sense that
\begin{align*}
    &\alpha\circ (\Delta\times_{B}\id_{P})\circ \delta=\alpha'\circ(\id_{\cL}\times_{B}\delta)\circ \delta:P\to \cL\times_{B} \cL \times_{B} P\\ 
    &(\varepsilon\diamond_{B}\cL)\circ \delta=\id_{P}.
\end{align*}
In addition, if $P$ is a $B$-ring such that
$$\gbeg4{5}\gvac1\got1{P}\gvac1\got1{P}\gnl
\glcm\glcm\gnl
\gcl1\gbr\gcl1\gnl
\gmu\gmu\gnl
\gob2{\cL}\gob2{P}
\gend=\gbeg4{5}\got1{P}\gvac1\got1{P}\gnl
\gwmu3\gnl
\glcm\gnl
\gcl1\gcl1\gnl
\gob1{\cL}\gob1{P}
\gend,$$
we call $P$ a left $\cL$-comodule algebra.

\subsection{Skew pairing between braided Hopf algebroids}
Similar to \cite{schau1}, we can construct skew pairing for Hopf algebroids. 
\begin{Def}
Let $\cL$ and $\cH$ be two braided $B$-bialgebroids in $\B$. A skew pairing is a map $\skp{-}{-}:\cL\ot\cH\to B$ in $\B$, such that
 {\allowdisplaybreaks\begin{gather}
     \label{Belinear}
           \gbeg45
       \got1B \got1{\cL}\got3{\cH}\gnl
        \glm\gvac1\gcl1\gnl
        \gvac1\glmpt\gnot{\skp{-}{-}}\gcmpb\grmpt\gnl
        \gvac1\gvac1\gcl1\gnl
        \gvac1\gob3B
        \gend
            =
             \gbeg45
       \got1B \got1{\cL}\got3{\cH}\gnl
        \gcl2\gcl1\gvac1\gcl1\gnl
        \gvac1\glmpt\gnot{\skp{-}{-}}\gcmpb\grmpt\gnl
        \gwmu3\gnl
        \gvac1\gob1B
        \gend\qquad\textup{and}\qquad\gbeg46
        \got1{\cL}\got1{B}\got1{\cH}\gnl
        \gcl1\gmp{\ob}\gcl1\gnl
        \gcl1\glm\gnl
        \glmpt\gnot{\skp{-}{-}}\gcmpb\grmpt\gnl
       \gvac1\gcl1\gnl
        \gob3B
        \gend
            =
             \gbeg46
        \got1{\cL}\gvac1\got1{B}\got1{\cH}\gnl
        \gcl1\gvac1\gbr\gnl
        \glmpt\gnot{\skp{-}{-}}\gcmpb\grmpt\gcl2\gnl
       \gvac1\gcl1\gnl
       \gvac1\gwmu3\gnl
        \gvac1\gob3B
        \gend
            \\\notag\\
     \label{adelli}
            \gbeg46
        \got1{\cL}\got1{B}\got1{\cH}\gnl
        \gcl1\gmp{\ob}\gcl1\gnl
        \grm\gcl1\gnl
        \glmpt\gnot{\skp{-}{-}}\gcmpb\grmpt\gnl
       \gvac1\gcl1\gnl
        \gob3B
        \gend
            =
             \gbeg46
        \got1{\cL}\gvac1\got1{B}\got1{\cH}\gnl
        \gcl1\gvac1\gibr\gnl
        \gcl1\gvac1\grm\gnl
        \glmpt\gnot{\skp{-}{-}}\gcmpb\grmpt\gnl
       \gvac1\gcl1\gnl
        \gob3B
        \gend
            \\\notag\\
     \label{agammanabla}
              \gbeg46
        \got1{\cL}\got1{B}\got1{\cH}\gnl
        \gcl1\gcl1\gcl1\gnl
        \grm\gcl1\gnl
        \glmpt\gnot{\skp{-}{-}}\gcmpb\grmpt\gnl
       \gvac1\gcl1\gnl
        \gob3B
        \gend
            =
             \gbeg48
        \got1{\cL}\got1{B}\got1{\cH}\gnl
        \gcl1\gcl1\gcl1\gnl
        \gibr\gcl1\gnl
        \gibr\gcl1\gnl
        \grm\gcl1\gnl
        \glmpt\gnot{\skp{-}{-}}\gcmpb\grmpt\gnl
       \gvac1\gcl1\gnl
        \gob3B
        \gend=  \gbeg46
        \got1{\cL}\got1{B}\got1{\cH}\gnl
        \gcl1\gcl1\gcl1\gnl
        \gcl1\glm\gnl
        \glmpt\gnot{\skp{-}{-}}\gcmpb\grmpt\gnl
       \gvac1\gcl1\gnl
        \gob3B
        \gend
            \\\notag\\
     \label{agammaeta}
            \gbeg48
        \gvac1\got1{B}\got1{\cL}\got1{\cH}\gnl
       \gvac1\gmp{\ob}\gcl1\gcl2\gnl
       \gvac1\gibr\gnl
       \gvac1\gcl1\gibr\gnl
       \gvac1\gcl1\grm\gnl
       \gvac1\gdnot{\skp{-}{-}}\glmptb\grmpt\gnl
        \gvac1\gcl1\gnl
        \gob3B
        \gend
            =
             \gbeg46
       \got1B \got1{\cL}\got3{\cH}\gnl
       \gmp{\ob}\gcl1\gvac1\gcl1\gnl
        \glm\gvac1\gcl1\gnl
        \gvac1\glmpt\gnot{\skp{-}{-}}\gcmpb\grmpt\gnl
        \gvac1\gvac1\gcl1\gnl
        \gvac1\gob3B
        \gend=  \gbeg48
        \gvac1\got1{B}\got1{\cL}\got1{\cH}\gnl
       \gvac1\gmp{\ob}\gcl1\gcl2\gnl
       \gvac1\gbr\gnl
       \gvac1\gcl1\gibr\gnl
       \gvac1\gcl1\grm\gnl
       \gvac1\gdnot{\skp{-}{-}}\glmptb\grmpt\gnl
        \gvac1\gcl1\gnl
        \gob3B
        \gend
        \\\notag\\
     \label{product}
             \gbeg58
       \got1{\cL} \got1{\cL}\got2{\cH}\gnl
       \gcl1\gcl1\gcmu\gnl
       \gcl1\gdnot{\skp{-}{-}}\glmptb\grmpt\gcl1\gnl
        \grm\gvac1\gcl1\gnl
        \gcn2113\gvac1\gcl1\gnl
        \gvac1\glmpt\gnot{\skp{-}{-}}\gcmpb\grmpt\gnl
        \gvac1\gvac1\gcl1\gnl
        \gvac1\gob3B
        \gend=\gbeg46
\got1{\cL}\got1{\cL}\got1{\cH}\gnl
\gmu\gcl1\gnl
\gcn2121\gcl1\gnl
\glmpt\gnot{\skp{-}{-}}\gcmpb\grmpt\gnl
\gvac1\gcl1\gnl
\gvac1\got1B
        \gend
            =
             \gbeg48
       \got1{\cL} \got1{\cL}\got2{\cH}\gnl
       \gcl1\gcl1\gcmu\gnl
       \gcl1\gdnot{\skp{-}{-}}\glmpt\grmptb\gcl1\gnl
        \gcl1\gvac1\glm\gnl
        \gcn2113\gvac1\gcl1\gnl
        \gvac1\glmpt\gnot{\skp{-}{-}}\gcmpb\grmpt\gnl
        \gvac1\gvac1\gcl1\gnl
        \gvac1\gob3B
        \gend
  \\\notag\\
     \label{coproduct}
            \gbeg59
       \got2{\cL} \got1{\cH}\got1{\cH}\gnl
       \gcmu\gbr\gnl
       \gcl1\gdnot{\skp{-}{-}}\glmptb\grmpt\gcl1\gnl\gcl1\gmp{\ob}\gvac1\gcl1\gnl
       \gibr\gvac1\gcl1\gnl
       \glm\gvac1\gcl1\gnl
       \gvac1\glmpt\gnot{\skp{-}{-}}\gcmpb\grmpt\gnl
       \gvac1\gvac1\gcl1\gnl
\gvac1\gvac1\gob1B
        \gend= \gbeg46
\got1{\cL}\got1{\cH}\got1{\cH}\gnl
\gcl1\gmu\gnl
\gcl1\gcn2123\gnl
\glmpt\gnot{\skp{-}{-}}\gcmpb\grmpt\gnl
\gvac1\gcl1\gnl
\gvac1\gob1B
        \gend
            =
             \gbeg49
       \got2{\cL} \got1{\cH}\got1{\cH}\gnl
       \gcmu\gcl1\gcl1\gnl
       \gcl1\gibr\gcl1\gnl
       \gcl1\gcl1\gdnot{\skp{-}{-}}\glmptb\grmpt\gnl
       \gcl1\gcl1\gmp{\ob}\gnl
       \gcl1\grm\gnl
       \gdnot{\skp{-}{-}}\glmptb\grmpt\gnl
       \gcl1\gnl
       \gob1B
        \gend
  \end{gather}}
\end{Def}

\subsection{Braided Ehresmann-Schauenburg Hopf algebroids}

In this section, we extend the construction of the Ehresmann--Schauenburg Hopf algebroid (ES-Hopf algebroid) associated to a Galois object \cite{schau2} to the setting of Hopf--Galois extensions. Furthermore, we will show that the resulting Ehresmann--Schauenburg Hopf algebroid  is a braided Hopf algebroid which generalizes the result in \cite[Remark 4.12]{HS25}.

\begin{Def}
  Let $H$ be a bialgebra  and $V$ a right $H$-comodule in $\B$.
  The coinvariant subobject $V^{co \,H}\subset V$ is the equalizer
  of $\delta,V\o\eta\colon V\to V\o H$.
  % The notation $\lcofix WH$ is used for the left
  % coinvariant subobject of a left $H$-comodule $W$.
\end{Def}

\begin{Def}
  Let $H$ be a Hopf algebra in $\B$, $P$ be a right $H$-comodule algebra in $\B$  with $B=P^{co \,H}\subset P$. If
  the canonical map
  $$\can=\gbeg3{5}\gvac1\got1{P\otimes_{B}P}\gnl
  \gcl1\gvac1\grcm\gnl
  \gcl1\gcn2131\gcl1\gnl
  \gmu\gvac1\gcl1\gnl
  \gob2P\gvac1\gob1H
  \gend
  $$
  is bijective,  we call $B\subset P$ a $H$-Galois extension. If $P$ is a faithfully flat $B$-module, we call $B\subseteq P$ a faithfully flat $H$-Galois extension.
\end{Def}
 If $B\subseteq P$ is a
$H$-Galois extension, the inverse of the canonical map can be determined by the right translation map:
    \begin{align*}
       \tau:=(\can)^{-1}|_{1\ot{}H}:H\to P\ot_{B}P, \qquad h\mapsto \tuno{h}\otimes\tdue{h},
    \end{align*}
    namely, $\can\inv(p\ot h)=p\tuno{h}\otimes\tdue{h}$. Graphically, we also have
     \begin{equation}\label{cantau}
  \gbeg46
  \got3H\gnl
  \gvac1\gcl1\gnl
  \glmpb\gnot{\tau}\gcmpt\grmpb\gnl
  \gcl1\gvac1\grcm\gnl
  \gwmu3\gcl1\gnl
  \gob3P\gob1H
  \gend
  =
  \gbeg25
  \gvac1\got1H\gnl
  \gvac1\gcl3\gnl
  \gu1\gnl
  \gcl1\gnl
  \gob1P\gob1{H.}
  \gend
 \qquad\textup{and}\qquad
  \gbeg35
  \got1P\gnl
  \grcm\gnl
  \gcl1\gwnot2\tau\glmptb\grmpb\gnl
  \gmu\gcl 1\gnl
  \gvac1\gob1{P\ot_{B}P}
  \gend
  =
  \gbeg 3{5}
  \gvac1\got1{B\ot_{B}P}\gnl
  \gcl1\gvac1\gcl3\gnl
  \gcl1\gvac1\gnl
  \gcl1\gvac1\gnl
  \gvac1\gob1{P\ot_{B}P}
  \gend=
  \gbeg3{5}\got1P\gnl
  \gcl3\gnl
  \gob1P\gend
  \end{equation}
 By the same proof (only need to take care of the balanced tensor product over $B$) in \cite[Lemma 3.4 ]{schau2}, we have
\begin{Lem}Let $B\subset P$ be a $H$-Galois extension in $\B$. Then
   {\allowdisplaybreaks\begin{gather}
     \label{delre}
            \gbeg 45
            \got 3H\gnl
            \gvac1\gcl1\gnl
            \glmpb\gnot\tau\gcmpt\grmpb\gnl
            \gcl1\gvac1\grcm\gnl
            \gvac1\gob1{P\ot_{B}P}\gob3{H}
            \gend
            =
            \gbeg45
            \gvac2\got1H\gnl
            \gvac1\gwcm3\gnl
            \glmpb\gnot\tau\gcmpt\grmpb\gcl2\gnl
            \gcl1\gvac1\gcl1\gnl
            \gvac1\gob1{P\ot_{B}P}\gob3{H}
            \gend
            \\\notag\\
     \label{delli}
            \gbeg46
            \got3H\gnl
            \gvac1\gcl1\gnl
            \glmpb\gnot\tau\gcmpt\grmpb\gnl
            \grcm\gcl1\gnl
            \gibr\gcn2113\gnl
            \gob1H\gob3{P\ot_{B}P}
            \gend
            =
            \gbeg46
            \got3H\gnl
            \gvac1\gcl1\gnl
            \gwcm3\gnl
            \gmp{S}\glmpb\gnot\tau\gcmpt\grmpb\gnl
            \gcl1\gcl1\gvac1\gcl1\gnl
            \gob1H\gob3{P\ot_{B}P}
            \gend
            \\\notag\\
     \label{gammanabla}
            \gbeg 4{6}
            \got1{B}\gvac1\got1{H}\gnl
            \gcl1
            \gvac1\gcl1\gnl
            \gcl1\glmpb\gnot{\tau}\gcmpt\grmpb\gnl
            \gcl1\gcl1\gvac1\gcl1\gnl
            \glm\gvac1\gcl1\gnl
            \gvac1\gvac1\gob1{P\ot_{B}P}
            \gend=\gbeg 4{6}
            \gvac1\got1{B}\gvac1\got1{H}\gnl
            \gvac1\gibbr31\gnl
            \glmpb\gnot{\tau}\gcmpt\grmpb\gcl1\gnl
            \gcl1\gvac1\gcl1\gcl1\gnl
            \gcl1\gvac1\grm\gnl
            \gvac1\gob1{P\ot_{B}P}
            \gend\qquad\textup{and}\qquad\gbeg 35
            \got1H\gvac1\got1H\gnl
            \gwmu3\gnl
            \glmpb\gnot\tau\gcmpt\grmpb\gnl
            \gcl1\gvac1\gcl1\gnl
            \gob3{P\ot_{B}P}
            \gend
            =
            \gbeg 47
            \got1H\gvac2\got1H\gnl
            \gcl1\gvac2\gcl1\gnl
            \gdnot\tau\glmptb\grmpb\gdnot\tau\glmpb\grmptb\gnl
            \gcl1\gbr\gcl1\gnl
            \gbr\gmu\gnl
            \gmu\gcn2122\gnl
             \gob4{P\ot_{B}P}
            \gend
            \\\notag\\
     \label{gammaeta}
            \gbeg 35
            \gnl
            \gvac1\gu1\gnl
            \glmpb\gnot\tau\gcmpt\grmpb\gnl
            \gcl1\gvac1\gcl1\gnl
            \gob3{P\ot_{B}P}
            \gend
            =
            \gbeg25
            \gvac1\gnl
            \gvac1\gnl
            \gu1\gvac1\gu1\gnl
            \gcl1\gvac1\gcl1\gnl
            \gob3{P\ot_{B}P}
            \gend\qquad\textup{and}\qquad\gbeg36\got3H\gnl
    \gvac1\gcl1\gnl
    \glmpb\gnot\tau\gcmpt\grmpb\gnl
    \gcl1\gvac1\gcl1\gnl
    \gwmu3\gnl
    \gob3P
    \gend=\gbeg16\got1H\gnl
    \gcl1\gnl
   \gcu1\gnl
    \gu1\gnl
    \gcl1\gnl
    \gob1P
    \gend.
  \end{gather}}
\end{Lem}
\begin{proof}
For the first equation of (\ref{gammanabla}). After applying the canonical map on both sides, we get the same result. For the second equation of (\ref{gammaeta}).
  $$\gbeg36\got3H\gnl
    \gvac1\gcl1\gnl
    \glmpb\gnot\tau\gcmpt\grmpb\gnl
    \gcl1\gvac1\gcl1\gnl
    \gwmu3\gnl
    \gob3P
    \gend=\gbeg46\gvac2\got1H\gnl
            \gvac1\gwcm3\gnl
            \glmpb\gnot\tau\gcmpt\grmpb\gcl1\gnl
            \gcl1\gvac1\gcl1\gcu1\gnl
    \gwmu3\gnl
    \gob3P
    \gend=\gbeg46\got3H\gnl
    \gvac1\gcl1\gnl
    \glmpb\gnot\tau\gcmpt\grmpb\gnl
    \gcl1\gvac1\grcm\gnl
    \gwmu3\gcu1\gnl
    \gob3P
    \gend=\gbeg16\got1H\gnl
    \gcl1\gnl
   \gcu1\gnl
    \gu1\gnl
    \gcl1\gnl
    \gob1P
    \gend.
    $$

\end{proof}

% \begin{Def}
% A Hopf module $M\in\BHM.H..P.$ for a right $H$-comodule algebra
% $P$ is a right $H$-comodule and right $P$-module such that the
% module structure on $M$ is an $H$-comodule morphism with respect
% to the codiagonal comodule structure on $M\o P$.
% \end{Def}
% By the same method as in \cite{schau2}, we have
% \begin{Prop}\nmlabel{Proposition}{structhm}
%   Assume that $B\subset P$ is a faithfully flat $H$-Galois extension in $\B$. Then for every Hopf
%   module in $\BHM.H..A.$ the morphism
%   $$\mu_0\colon\rcofix MH\o A\to M$$
%   is an isomorphism. The inverse is determined by commutativity of
%   $$\xymatrix{M\ar[r]^-{\mu_0\inv}\ar[d]_{\delta}
%       &\rcofix MH\o A\ar[dd]^{\iota\o A}\\
%       M\o H\ar[d]_{M\o\gamma}
%       &
%       \\
%       M\o A\o A\ar[r]^-{\mu\o A}
%       &M\o A.}$$
% \end{Prop}

\begin{Prop}\nmlabel{Proposition}{coHsubalg}
  Let $B\subset P$ be a  $H$-Galois extension in $\B$.
  Then the right coinvariant subobject $L(P,H):=\rcofix{(P\o P)}H$ taken with respect
  to the codiagonal comodule structure is a subalgebra of $P\o \overline{P}$. Moreover, $\rcofix{(P\o P)}H$ is a $B\ot\BB$-ring with $\eta:B^e\to \rcofix{(P\o P)}H, a\ot\ol b\mapsto a\ot \ol b$. 
\end{Prop}
\begin{proof}The proof is the same as in \cite[Proposition 2.2.]{schau2}. We only need to observe that $\eta$ is an algebra map.
\end{proof}
Since $P$ is a $B$-ring and $\ol{P}$ is a $\BB$-ring, we observe that the left $B$ and $\BB$-actions on $L=L(P,H)$ are given by
$$\gbeg23\got1B\got1L\gnl\glm\gnl\gvac1\gob1L\gend=\gbeg66
      \got1{B}\gvac1\got1{L}\gnl
    \gcl2 \gvac1\gcl1\gnl
     \gvac1\glmpb\gnot\iota\gcmpt\grmpb\gnl
     \glm\gvac1\gcl1\gnl
     \gvac1\gcl1\gvac1\gcl1\gnl\gvac1\gvac1\gob1{L}\gend\qquad\textup{and}\qquad\gbeg23\got1{\BB}\got1L\gnl\glm\gnl\gvac1\gob1L\gend=\gbeg67
      \got1{\BB}\gvac1\got1{L}\gnl
    \gcl2 \gvac1\gcl1\gnl
     \gvac1\glmpb\gnot\iota\gcmpt\grmpb\gnl
     \gbr\gvac1\gcl1\gnl
     \gcl1\gcn2113\gcl1\gnl
     \gcl1\gvac1\gmu\gnl
     \gvac1\gvac1\gnl\gvac1\gob1{L}\gend=\gbeg6{10}
      \got1{\BB}\got1{L}\gnl
    \gcl1\gcl1\gnl
     \gmp{\ob}\gdnot\iota\glmptb\grmpb\gnl
     \gbr\gmp{\ob}\gnl
     \gcl1\gbr\gnl
     \gcl1\gmu\gnl
     \gcl1\gcn2123\gnl
     \gcl1\gvac1\gmp{\ob}\gnl
     \gcl1\gvac1\gcl1\gnl
     \gvac1\gvac1\gnl\gvac1\gob1{L}\gend,$$
      where $\iota:L\to P\ot \ol{P}$ is the inclusion map. Notice that if we identify $P$ with $\ol P$, the multiplication of $L$ is
     $$ \gbeg2{3}\got1L\got1L\gnl
      \gmu\gnl
      \gob2L
      \gend=\gbeg47
  \got1L\got3L\gnl
  \gcl1\gvac1\gcl1\gnl
  \gdnot\iota\glmptb\grmpb\gdnot\iota\glmptb\grmpb\gnl
  \gcl1\gbr\gcl1\gnl
  \gmu\gbr\gnl
  \gcn2122\gmu\gnl
  \gob2P\gob2P
  \gend$$
  As a generalization of \cite{schau2}, we have
\begin{Lem}
  Let $B\subset P$ be a faithfully flat $H$-Galois extension in $\B$.
  Then $L=L(P,H)$ is a $B$-bialgebroid in $\B$. More precisely, the $B^e$-ring structure is given by  Proposition~\ref{coHsubalg},  the coproduct and counit are given by
  $$\Delta=\gbeg68
    \got3{L}\gnl
    \gvac1\gcl1\gnl
    \glmpb\gnot\iota\gcmpt\grmpb\gnl
    \grcm\gcl1\gnl
    \gcl1\gcn1113\gcn2115\gnl
    \gcl1\glmpb\gnot\tau\gcmpt\grmpb\gcn2111\gnl
    \gcl1\gcl1\gvac1\gcl1\gcl1\gnl
    \gob2{L}\gvac1\gob2{L}\gend\qquad\textup{and}\qquad\varepsilon=\gbeg66
    \got3{L}\gnl
    \gvac1\gcl1\gnl
    \glmpb\gnot\iota\gcmpt\grmpb\gnl
    \gcl1\gvac1\gcl1\gnl
    \gwmu3\gnl
\gvac1\gob1B
    \gend.$$
\end{Lem}
\begin{proof}
By the first equation of (\ref{gammanabla}), we can see $\Delta$ is well defined. Also, it is left $B$ and left $\BB$-linear. Now, let's check the image of $\Delta$ belongs to the Takeuchi product. Indeed,
$$\gbeg6{11}
    \got3{L}\gvac1\gvac1\got1B\gnl
    \gvac1\gcl1\gvac1\gvac1\gvac1\gcl2\gnl
    \glmpb\gnot\iota\gcmpt\grmpb\gnl
    \grcm\gcl1\gvac1\gvac1\gmp{\ob}\gnl
    \gcl1\gcn1113\gcn2115\gvac1\gcl2\gnl
    \gcl1\glmpb\gnot\tau\gcmpt\grmpb\gcn2111\gnl
    \gcl1\gcl1\gvac1\gcl1\gbr\gnl
    \gcl1\gcl1\gvac1\gbr\gcl1\gnl
    \gcl1\gcl1\gcn2131\gcl1\gcl1\gnl
     \gcl1\gmu\gvac1\gcl1\gcl1\gnl
    \gob2{L}\gvac1\gvac1\gob2{L}\gend=\gbeg6{14}
    \got3{L}\gvac1\gvac1\got1B\gnl
    \gvac1\gcl1\gvac1\gvac1\gvac1\gcl2\gnl
    \glmpb\gnot\iota\gcmpt\grmpb\gnl
    \grcm\gcl1\gvac1\gvac1\gcl1\gnl
    \gcl1\gcn1113\gcn2115\gvac1\gcl2\gnl
    \gcl1\glmpb\gnot\tau\gcmpt\grmpb\gcn2111\gnl
    \gcl1\gcl1\gvac1\gcl1\gbr\gnl
    \gcl1\gmp{\ob}\gvac1\gbr\gcl1\gnl
    \gcl1\gcl1\gcn2131\gcl1\gcl1\gnl
     \gcl1\gbr\gvac1\gcl1\gcl1\gnl
     \gcl1\gmu\gvac1\gcl1\gcl1\gnl
     \gcl1\gcn2123\gvac1\gcl1\gcl1\gnl
     \gcl1\gvac1\gmp{\ob}\gvac1\gcl1\gcl1\gnl
    \gvac1\gob2{L}\gvac1\gob2{L}\gend=\gbeg6{9}
    \got3{L}\gvac1\gvac1\got1B\gnl
    \gvac1\gcl1\gvac1\gvac1\gvac1\gcl2\gnl
    \glmpb\gnot\iota\gcmpt\grmpb\gnl
    \grcm\gcl1\gvac1\gvac1\gcl1\gnl
    \gcl1\gcn1113\gcn2115\gvac1\gcl2\gnl
    \gcl1\glmpb\gnot\tau\gcmpt\grmpb\gcn2111\gnl
    \gcl1\gcl1\gvac1\gcl1\gbr\gnl
 \gcl1\gcl1\gvac1\gmu\gcl1\gnl
    \gob2{L}\gvac1\gvac1\gob2{L}\gend$$
    where the first step use the definition of opposite algebra, the second step uses the first equation of (\ref{gammanabla}). We show that the first leg belongs to $L$. Indeed
$$\gbeg6{11}
   \got3{L}\gvac1\gnl
   \gvac1\gcl1\gnl
\glmpb\gnot\iota\gcmpt\grmpb\gnl
\grcm\gcn2115\gnl
\gcl1\gcn2113\gvac1\gcn2113\gnl
\gcl1\glmpb\gnot\tau\gcmpt\grmpb\gvac1\gcl1\gnl
\gcl1\gcn2113\gcn2113\gcl1\gnl
\grcm\grcm\gcl1\gcl1\gnl
\gcl1\gbr\gcl1\gcl1\gcl1\gnl
\gcl1\gcl1\gmu\gcl1\gcl1\gnl
\gob2{L}\gob2{H}\gvac1\gob1{L}
\gend
=\gbeg6{15}
   \got3{L}\gvac1\gnl
   \gvac1\gcl1\gnl
\glmpb\gnot\iota\gcmpt\grmpb\gnl
\grcm\gcn2113\gnl
\gcl1\gcn2113\gcn2113\gnl
\gcl1\gwcm3\gcn2113\gnl
\gcl1\gcl1\gvac1\gdnot\tau\glmptb\grmpb\gcl1\gnl
\gcl1\gcl1\gcn2131\gcl1\gcl1\gnl
\gcl1\gcl1\grcm\gcl1\gcl1\gnl
\gcl1\gcl1\gibr\gcl1\gcl1\gnl
\gcl1\gmu\gcl1\gcl1\gcl1\gnl
\gcl1\gcn2123\gcl1\gcl1\gcl1\gnl
\gcl1\gvac1\gbr\gcl1\gcl1\gnl
\gcl1\gvac1\gcl1\gcl1\gcl1\gcl1\gnl
\gvac1\gob1L\gvac1\gob1{H}\gob2{L}
\gend
=\gbeg6{14}
   \got3{L}\gvac1\gnl
   \gvac1\gcl1\gnl
\glmpb\gnot\iota\gcmpt\grmpb\gnl
\grcm\gcn2113\gnl
\gcl1\gcn2112\gcn2113\gnl
\gcl1\gcmu\gvac1\gcl1\gnl
\gcl1\gcl1\gcn2112\gcl1\gnl
\gcl1\gcl1\gcmu\gcn2113\gnl
\gcl1\gcl1\gmp{S}\gdnot\tau\glmptb\grmpb\gcl1\gnl
\gcl1\gmu\gcl1\gcl1\gcl1\gnl
\gcl1\gcn2123\gcl1\gcl1\gcl1\gnl
\gcl1\gvac1\gbr\gcl1\gcl1\gnl
\gcl1\gvac1\gcl1\gcl1\gcl1\gcl1\gnl
\gvac1\gob1L\gvac1\gob1{H}\gob2{L}
\gend=
\gbeg6{8}
   \got3{L}\gvac1\gnl
   \gvac1\gcl1\gnl
\glmpb\gnot\iota\gcmpt\grmpb\gnl
\grcm\gcn2115\gnl
\gcl1\gcn2113\gvac1\gcn2113\gnl
\gcl1\glmpb\gnot\tau\gcmpt\grmpb\gvac1\gcl1\gnl
\gcl1\gcl1\gu1\gcl1\gvac1\gcl1\gnl
\gob2{L}\gob1{H}\gvac1\gob1{L}
\gend$$
where the 2nd step uses (\ref{delli}).  For the second leg, we have
$$\gbeg6{11}
   \got3{L}\gvac1\gnl
   \gvac1\gcl1\gnl
\glmpb\gnot\iota\gcmpt\grmpb\gnl
\grcm\gcn2115\gnl
\gcl1\gcn2113\gvac1\gcn2113\gnl
\gcl1\glmpb\gnot\tau\gcmpt\grmpb\gvac1\gcl1\gnl
\gcl1\gcl1\gvac1\grcm\grcm\gnl
\gcl1\gcl1\gvac1\gcl1\gbr\gcl1\gnl
\gcl1\gcl1\gvac1\gcl1\gcl1\gmu\gnl
\gcl1\gcl1\gvac1\gcl1\gcl1\gcn2122\gnl
\gob2{L}\gvac1\gob2{L}\gob2H
\gend
=
\gbeg6{10}
   \got3{L}\gvac1\gnl
   \gvac1\gcl1\gnl
\glmpb\gnot\iota\gcmpt\grmpb\gnl
\grcm\gcn2115\gnl
\gcl1\gcn2113\gvac1\gcn2113\gnl
\gcl1\gwcm3\gvac1\gcl1\gnl
\gcl1\gdnot\tau\glmptb\grmpb\gcn2113\grcm\gnl
\gcl1\gcl1\gcl1\gvac1\gbr\gcl1\gnl
\gcl1\gcl1\gcl1\gvac1\gcl1\gmu\gnl
\gob2{L}\gvac1\gob2{L}\gob2H
\gend
=
\gbeg5{10}
   \got3{L}\gvac1\gnl
   \gvac1\gcl1\gnl
\glmpb\gnot\iota\gcmpt\grmpb\gnl
\grcm\gcn2113\gnl
\gcl1\gcn2113\gcn2113\gnl
\grcm\gcn2113\grcm\gnl
\gcl1\gdnot\tau\glmptb\grmpb\gbr\gcl1\gnl
\gcl1\gcl1\gcl1\gcl1\gmu\gnl
\gcl1\gcl1\gcl1\gcl1\gcn2122\gnl
\gob2{L}\gob2{L}\gob2H
\gend=
\gbeg68
    \got3{L}\gnl
    \gvac1\gcl1\gnl
    \glmpb\gnot\iota\gcmpt\grmpb\gnl
    \grcm\gcl1\gnl
    \gcl1\gcn1111\gcn2113\gnl
    \gcl1\gdnot\tau\glmptb\grmpb\gcl1\gnl
    \gcl1\gcl1\gcl1\gcl1\gu1\gnl
    \gob2L\gob2L\gob1H
    \gend$$
 where the 2nd step uses  (\ref{delre}).        
We show that the coproduct is coassociative.

$$\gbeg6{10}
   \got3{L}\gvac1\gnl
   \gvac1\gcl1\gnl
\glmpb\gnot\iota\gcmpt\grmpb\gnl
\grcm\gcn2113\gnl
\gcl1\gdnot\tau\glmptb\grmpb\gcn2113\gnl
\gcl1\gcl1\grcm\gcn2113\gnl
\gcl1\gcl1\gcl1\gdnot\tau\glmptb\grmpb\gcl1\gnl
\gcl1\gcl1\gcl1\gcl1\gcl1\gcl1\gnl
\gend
=
\gbeg6{10}
   \got3{L}\gvac1\gnl
   \gvac1\gcl1\gnl
\glmpb\gnot\iota\gcmpt\grmpb\gnl
\grcm\gcn2113\gnl
\gcl1\gcn2113\gcn2113\gnl
\gcl1\gwcm3\gcn2113\gnl
\gcl1\gdnot\tau\glmptb\grmpb\gdnot\tau\glmptb\grmpb\gcl1\gnl
\gcl1\gcl1\gcl1\gcl1\gcl1\gcl1\gnl
\gend=
\gbeg8{10}
   \got3{L}\gvac1\gnl
   \gvac1\gcl1\gnl
\glmpb\gnot\iota\gcmpt\grmpb\gnl
\grcm\gcn2115\gnl
\gcl1\gcn2113\gvac1\gcn2113\gnl
\grcm\gcn2113\gvac1\gcl1\gnl
\gcl1\gdnot\tau\glmptb\grmpb\gdnot\tau\glmptb\grmpb\gcl1\gnl
\gcl1\gcl1\gcl1\gcl1\gcl1\gcl1\gnl
\gend$$
  where the 2nd step uses (\ref{delre}).  
We show that the coproduct is an algebra map. 
$$\gbeg8{10}
   \got3{L}\gvac1\gvac1\got1{L}\gnl
   \gvac1\gcl1\gvac1\gvac1\gvac1\gcl1\gnl
\glmpb\gnot\iota\gcmpt\grmpb\gvac1\glmpb\gnot\iota\gcmpt\grmpb\gnl
\grcm\gcn2113\grcm\gcn2113\gnl
\gcl1\gdnot\tau\glmptb\grmpb\gbr \gdnot\tau\glmptb\grmpb\gcl1\gnl
\gcl1\gcl1\gbr\gbr\gcl1\gcl1\gnl
\gcl1\gbr\gbr\gbr\gcl1\gnl
\gmu\gbr\gmu\gbr\gnl
\gcn2122\gmu\gcn2122\gmu\gnl
    \gvac1\gob2{L}\gvac1\gvac1\gob2{L}\gend
    =\gbeg8{12}
   \gvac1\got3{L}\gvac1\gvac1\got1{L}\gnl
   \gvac1\gvac1\gcl1\gvac1\gvac1\gvac1\gcl1\gnl
\gvac1\glmpb\gnot\iota\gcmpt\grmpb\gvac1\glmpb\gnot\iota\gcmpt\grmpb\gnl
\gvac1\grcm\gbbr31\gvac1\gcl1\gnl
\gvac1\gcl1\gcl1\grcm\gcn2113\gcl1\gnl
\gcn2131\gbr\gdnot\tau\glmptb\grmpb\gbr\gnl
\gcl1\gcn2131\gcl1\gcl1\gcl1\gmu\gnl
\gmu\gdnot\tau\glmpb\grmptb\gcl1\gcl1\gcn2122\gnl
\gcn2122\gcl1\gbr\gcl1\gcn2122\gnl
\gcn2122\gbr\gmu\gcn2122\gnl
\gcn2122\gmu\gcn2122\gcn2122\gnl
\gvac1\gob2{L}\gvac1\gvac1\gob2{L}
\gend
=
\gbeg6{11}
 \gvac1\got3{L}\gvac1\gvac1\got1{L}\gnl
   \gvac1\gvac1\gcl1\gvac1\gvac1\gvac1\gcl1\gnl
\gvac1\glmpb\gnot\iota\gcmpt\grmpb\gvac1\glmpb\gnot\iota\gcmpt\grmpb\gnl
\gvac1\grcm\gbbr31\gvac1\gcl1\gnl
\gvac1\gcl1\gcl1\grcm\gcn2113\gcl1\gnl
\gvac1\gcl1\gbr\gcl1\gvac1\gcl1\gcl1\gnl
\gvac1\gmu\gmu\gvac1\gbr\gnl
\gvac1\gcn2122\gcn2123\gvac1\gmu\gnl
\gvac1\gcn2122\glmpb\gnot\iota\gcmpt\grmpb\gcn2122\gnl
\gvac1\gcn2122\gcl1\gvac1\gcl1\gcn2122\gnl
\gvac1\gvac1\gob2{L}\gvac1\gvac1\gob2{L}
\gend
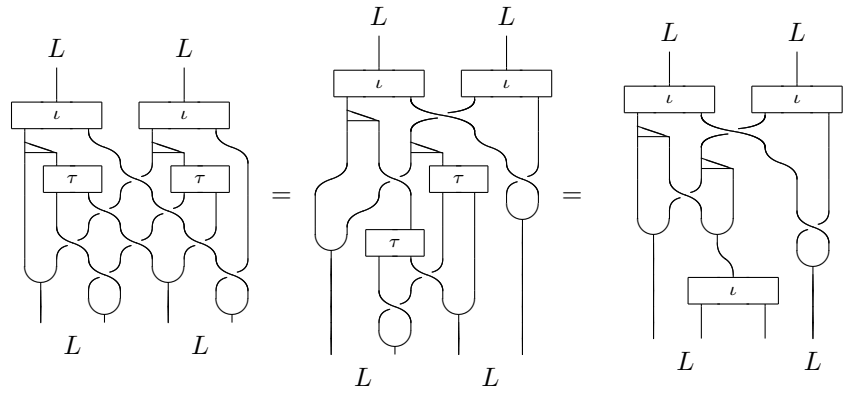
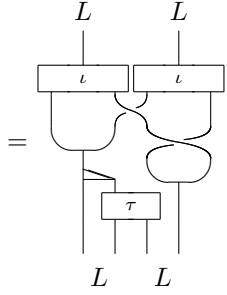
$$
$$=
\gbeg6{9}
\got3{L}\gvac1\got1{L}\gnl
\gvac1\gcl1\gvac1\gvac1\gcl1\gnl
\glmpb\gnot\iota\gcmpt\grmpb\glmpb\gnot\iota\gcmpt\grmpb\gnl
\gcl1\gvac1\gbr\gvac1\gcl1\gnl
\gwmu3\gbbr31\gnl
\gvac1\grcm\gwmu3\gnl
\gvac1\gcl1\gdnot\tau\glmptb\grmpb\gcl1\gnl
\gvac1\gcl1\gcl1\gcl1\gcl1\gnl
\gvac1\gob2{L}\gob2{L}
\gend$$
where the 2nd step uses (\ref{gammanabla}).   Now, we check the axiom for counit. First, we can see $\varepsilon$ is left $B^e$-linear.
    $$(\varepsilon\diamond_{B}\id)\circ\Delta=\gbeg6{10}
    \got3{L}\gnl
    \gvac1\gcl1\gnl
    \glmpb\gnot\iota\gcmpt\grmpb\gnl
    \grcm\gcl1\gnl
    \gcl1\gcn1113\gcn2115\gnl
    \gcl1\glmpb\gnot\tau\gcmpt\grmpb\gcn2111\gnl
    \gcl1\gcl1\gvac1\gcl1\gcl1\gnl
   \gmu\gvac1\gcl1\gcl1\gnl
   \gwmuh427\gcl1\gnl
   \gvac1\gvac1\gob3{L}
    \gend=\gbeg5{7}
    \got3{L}\gnl
    \gvac1\gcl1\gnl
    \glmpb\gnot\iota\gcmpt\grmpb\gnl
    \gcl3\gvac1\gcl3\gnl
    \gob3{L}\gend
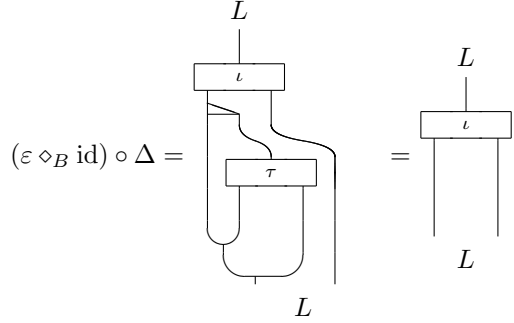
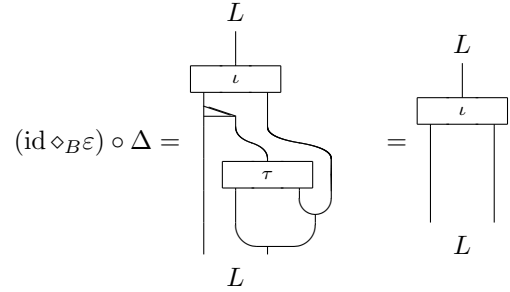
$$
    $$(\id\diamond_{B}\varepsilon)\circ\Delta=\gbeg6{9}
    \got3{L}\gnl
    \gvac1\gcl1\gnl
    \glmpb\gnot\iota\gcmpt\grmpb\gnl
    \grcm\gcl1\gnl
    \gcl1\gcn1113\gcn2115\gnl
    \gcl1\glmpb\gnot\tau\gcmpt\grmpb\gcn2111\gnl
    \gcl1\gcl1\gvac1\gmu\gnl
    \gcl1\gwmuh316\gnl
   \gvac1\gob1{L}
    \gend=\gbeg5{7}
    \got3{L}\gnl
    \gvac1\gcl1\gnl
    \glmpb\gnot\iota\gcmpt\grmpb\gnl
    \gcl3\gvac1\gcl3\gnl
    \gob3{L}\gend$$
    for both of them we use (\ref{gammaeta}).
    It is not hard to show $\varepsilon(X\varepsilon(Y))=\varepsilon(XY)=\varepsilon(X\overline{\varepsilon(Y)})$, for any $X,Y\in L$.
    
\end{proof}
\begin{Rem}
From the above proof we can see that $P$ is a left $L(P, H)$-comodule algebra in $\B$ with the coaction.
\[\delta(p)=p\z\ot p\on\tuno{}\ot p\on\tdue{}\in L(P,H)\times_{B}P.\]
In \cite[Theorem 4.3]{schau2}, there is a proof using a different method showing that the coproduct is coassociative and an algebra map, although it is for the `non-oid' case of a Galois object. Here we prove it in a more concrete way in order to help observe the balanced tensor product.
\end{Rem}

\begin{Thm}\label{thm. braided ES Hopf algebroid}
Let $B\subset P$ be a faithfully flat $H$-Galois extension in $\B$.
  Then $L=L(P,H)$ is a left $B$-Hopf algebroid in $\B$. In addition, if the antipode of $H$ is bijective, then $L$ is a $B$-Hopf algebroid in $\B$.  More precisely, the translation maps are given by
  $$\gbeg 45
            \got 3{L}\gnl
            \gvac1\gcl1\gnl
            \glmpb\gnot{+-}\gcmpt\grmpb\gnl
            \gcl1\gvac1\gcl1\gnl
            \gob3{L\otimes_{\BB}L}
            \gend=\gbeg 4{10}
            \got 3{L}\gnl
            \gvac1\gcl1\gnl
            \glmpb\gnot{\iota}\gcmpt\grmpb\gnl
            \gcl1\gvac1\grcm\gnl
            \gcl1\gvac1\gcl1\gcn2113\gnl
            \gcl1\gvac1\gcl1\glmpb\gnot{\tau}\gcmpt\grmpb\gnl
            \gcl1\gvac1\gcl1\gibbr31\gnl
            \gcl1\gvac1\gibr\gvac1\gcl1\gnl
            \gcl1\gvac1\gcl1\gcl1\gvac1\gcl1\gnl
            \gvac1\gvac1\gvac1\gob1{L\ot_{\BB}L}
            \gend$$
             $$\gbeg 45
            \got 3{L}\gnl
            \gvac1\gcl1\gnl
            \glmpb\gnot{[+][-]}\gcmpt\grmpb\gnl
            \gcl1\gvac1\gcl1\gnl
            \gob3{L\otimes_{B}L}
            \gend=\gbeg 4{11}
            \gvac1\gvac1\got 3{L}\gnl
            \gvac1\gvac1\gvac1\gcl1\gnl
            \gvac1\gvac1\glmpb\gnot{\iota}\gcmpt\grmpb\gnl
            \gvac1\gvac1\grcm\gcl1\gnl
            \gvac1\gvac1\gibr\gcl1\gnl
            \gvac1\gcn1131\gvac1\gibr\gnl
            \gvac1\gmp{\hat{S}}\gvac1\gcl1\gcl1\gnl
            \glmpb\gnot{\tau}\gcmpt\grmpb\gcl1\gcl1\gnl
            \gcl1\gvac1\gibr\gcl1\gnl
            \gcl1\gvac1\gcl1\gcl1\gcl1\gnl
\gvac1\gvac1\gob1{L\ot_{B}L}
            \gend$$
\end{Thm}
\begin{proof}
We only check the anti-translation map. To check the left leg is $H$-coinvariant, we have
$$\gbeg 5{14}
            \gvac1\gvac1\got 3{L}\gnl
            \gvac1\gvac1\gvac1\gcl1\gnl
            \gvac1\gvac1\glmpb\gnot{\iota}\gcmpt\grmpb\gnl
            \gvac1\gvac1\grcm\gcl1\gnl
            \gvac1\gvac1\gibr\gcl1\gnl
            \gvac1\gcn1131\gvac1\gibr\gnl
            \gvac1\gmp{\hat{S}}\gvac1\gcl1\gcl1\gnl
            \glmpb\gnot{\tau}\gcmpt\grmpb\gcl1\gcl1\gnl
            \gcl1\gvac1\gibr\gcl1\gnl
            \grcm\gcl1\gcn1213\gcn1213\gnl
            \gcl1\gcl1\grcm\gnl
            \gcl1\gbr\gcl1\gcl1\gcl1\gnl
            \gcl1\gcl1\gmu\gcl1\gcl1\gnl
            \gend
            =\gbeg 6{15}
            \gvac1\gvac1\got 3{L}\gnl
            \gvac1\gvac1\gvac1\gcl1\gnl
            \gvac1\gvac1\glmpb\gnot{\iota}\gcmpt\grmpb\gnl
            \gvac1\gvac1\grcm\gcn2113\gnl
            \gvac1\gvac1\gibr\gvac1\gcl1\gnl
            \gvac1\gcn1131\gvac1\gcn2113\gcl1\gnl \gvac1\gmp{\hat{S}}\gvac1\gvac1\gibr\gnl           \gwcm3\gvac1\gcl1\gcl1\gnl         \gcl1\glmpb\gnot{\tau}\gcmpt\grmpb\gcl1\gcl1\gnl
            \gmp{S}\gcl1\gvac1\gibr\gcl1\gnl
            \gbr\gvac1\gcl1\gcn1213\gcn1213\gnl
            \gcl1\gcn2113\grcm\gnl
            \gcl1\gvac1\gbr\gcl1\gcl1\gcl1\gnl
            \gcl1\gvac1\gcl1\gmu\gcl1\gcl1\gnl
            \gend=\gbeg 6{15}
            \gvac1\gvac1\got 3{L}\gnl
            \gvac1\gvac1\gvac1\gcl1\gnl
            \gvac1\gvac1\glmpb\gnot{\iota}\gcmpt\grmpb\gnl
            \gvac1\gvac1\grcm\gcn2113\gnl
            \gvac1\gvac1\gibr\gvac1\gcl1\gnl
            \gvac1\gcn1131\gvac1\gcn2113\gcl1\gnl \gwcm3\gvac1\gibr\gnl           \gibbr31\gvac1\gcl1\gcl1\gnl   \gmp{\hat{S}}\gvac1\gmp{\hat{S}}\gvac1\gcl1\gcl1\gnl      \gcl1\glmpb\gnot{\tau}\gcmpt\grmpb\gcl1\gcl1\gnl
            \gmp{S}\gcl1\gvac1\gibr\gcl1\gnl
            \gbr\gvac1\gcl1\gcn1213\gcn1213\gnl
            \gcl1\gcn2113\grcm\gnl
            \gcl1\gvac1\gbr\gcl1\gcl1\gcl1\gnl
            \gcl1\gvac1\gcl1\gmu\gcl1\gcl1\gnl
            \gend=\gbeg 6{16}
            \gvac1\gvac1\got 3{L}\gnl
            \gvac1\gvac1\gvac1\gcl1\gnl
            \gvac1\gvac1\glmpb\gnot{\iota}\gcmpt\grmpb\gnl
            \gvac1\gvac1\grcm\gcn2113\gnl
            \gvac1\gvac1\gcl1\gcn2113\gcl1\gnl
            \gvac1\gvac1\grcm\gcl1\gcl4\gnl
            \gvac1\gvac1\gibr\gcl1\gnl
            \gvac1\gcn1131\gvac1\gibr\gnl
            \gvac1\gmp{\hat{S}}\gvac1\gcl1\gcl1\gnl
            \glmpb\gnot{\tau}\gcmpt\grmpb\gcl1\gibr\gnl
            \gcl1\gvac1\gibr\gcl1\gcl1\gnl
            \gcl1\gvac1\gcl1\gibr\gcn2113\gnl
            \gcl1\gvac1\gcl1\gcl1\gcn2113\gcn2111\gnl
            \gcl1\gvac1\gcl1\grcm\gcl1\gcl1\gnl
            \gcl1\gvac1\gbr\gcl1\gcl1\gcl1\gnl
            \gcl1\gvac1\gcl1\gmu\gcl1\gcl1\gnl
            \gend=\gbeg 5{14}
            \gvac1\gvac1\got 3{L}\gnl
            \gvac1\gvac1\gvac1\gcl1\gnl
            \gvac1\gvac1\glmpb\gnot{\iota}\gcmpt\grmpb\gnl
            \gvac1\gvac1\grcm\gcl1\gnl
            \gvac1\gvac1\gibr\gcl1\gnl
            \gvac1\gcn1131\gvac1\gibr\gnl
            \gvac1\gmp{\hat{S}}\gvac1\gcl1\gcl1\gnl
            \glmpb\gnot{\tau}\gcmpt\grmpb\gcl1\gcl1\gnl
            \gcl1\gvac1\gibr\gcl1\gnl
            \gcl1\gvac1\gcl1\gcn1113\gcn1113\gnl
            \gcl1\gvac1\gcl1\gu1\gcl1\gcl1\gnl
            \gend$$
            For the other leg, we have 
            $$\gbeg 6{15}
            \gvac1\gvac1\got 3{L}\gnl
            \gvac1\gvac1\gvac1\gcl1\gnl
            \gvac1\gvac1\glmpb\gnot{\iota}\gcmpt\grmpb\gnl
            \gvac1\gvac1\grcm\gcl1\gnl
            \gvac1\gvac1\gibr\gcl1\gnl
            \gvac1\gcn1131\gvac1\gibr\gnl
            \gvac1\gmp{\hat{S}}\gvac1\gcl1\gcl1\gnl
            \glmpb\gnot{\tau}\gcmpt\grmpb\gcl1\gcl1\gnl
            \gcl1\gvac1\gibr\gcn2113\gnl
            \gcl1\gvac1\gcl1\grcm\grcm\gnl
            \gcl1\gvac1\gcl1\gcl1\gbr\gcl1\gnl
            \gcl1\gvac1\gcl1\gcl1\gcl1\gmu\gnl
            \gend=
            \gbeg 6{15}
            \gvac1\gvac1\gvac1\got 3{L}\gnl
            \gvac1\gvac1\gvac1\gvac1\gcl1\gnl
            \gvac1\gvac1\gvac1\glmpb\gnot{\iota}\gcmpt\grmpb\gnl
            \gvac1\gvac1\gvac1\grcm\gcl1\gnl
            \gvac1\gvac1\gvac1\gibr\gcl1\gnl
            \gvac1\gvac1\gcn1131\gvac1\gibr\gnl
            \gvac1\gvac1\gmp{\hat{S}}\gvac1\gcl1\gcl1\gnl
            \gvac1\gwcm3\gcl1\gcl1\gnl
            \gvac1\gcl1\gvac1\gibr\gcl1\gnl
            \glmpb\gnot{\tau}\gcmpt\grmpb\gcl1\gcl1\gcl1\gnl
            \gcl1\gvac1\gibr\gcl1\gcl1\gnl
            \gcl1\gvac1\gcl1\gcl1\gcl1\grcm\gnl
            \gcl1\gvac1\gcl1\gcl1\gbr\gcl1\gnl
            \gcl1\gvac1\gcl1\gcl1\gcl1\gmu\gnl
            \gend= \gbeg 6{15}
            \gvac1\gvac1\gvac1\got 3{L}\gnl
            \gvac1\gvac1\gvac1\gvac1\gcl1\gnl
            \gvac1\gvac1\gvac1\glmpb\gnot{\iota}\gcmpt\grmpb\gnl
            \gvac1\gvac1\gvac1\grcm\gcl1\gnl
            \gvac1\gvac1\gvac1\gibr\gcl1\gnl
            \gvac1\gvac1\gcn1131\gvac1\gibr\gnl
            \gvac1\gwcm3\gcl1\gcl1\gnl
            \gvac1\gibbr31\gcl1\gcl1\gnl
            \gvac1\gmp{\hat{S}}\gvac1\gmp{\hat{S}}\gcl1\gcl1\gnl
            \gvac1\gcl1\gvac1\gibr\gcl1\gnl
            \glmpb\gnot{\tau}\gcmpt\grmpb\gcl1\gcl1\gcl1\gnl
            \gcl1\gvac1\gibr\gcl1\gcl1\gnl
            \gcl1\gvac1\gcl1\gcl1\gcl1\grcm\gnl
            \gcl1\gvac1\gcl1\gcl1\gbr\gcl1\gnl
            \gcl1\gvac1\gcl1\gcl1\gcl1\gmu\gnl
            \gend
            =\gbeg6{16}\gvac1\gvac1\gvac1\got 3{L}\gnl
            \gvac1\gvac1\gvac1\gvac1\gcl1\gnl
            \gvac1\gvac1\gvac1\glmpb\gnot{\iota}\gcmpt\grmpb\gnl
            \gvac1\gvac1\gvac1\grcm\gcl1\gnl
            \gvac1\gvac1\gcn2131\gmp{\hat{S}}\gcl1\gnl
            \gvac1\gvac1\grcm\gcl1\gcl1\gnl
            \gvac1\gcn2131\gmp{\hat{S}}\gcl1\gcl1\gnl
            \gvac1\grcm\gcl1\gcl1
            \gcl1\gnl
            \gvac1\gcl1\gibr\gcl1\gcl1\gnl
            \gvac1\gcl1\gmu\gcl1\gcl1\gnl
            \gvac1\gcl1\gcn2123\gcl1\gcl1\gnl
            \gvac1\gcl1\gvac1\gibr\gcl1\gnl
            \gvac1\gibbr31\gcl1\gcl1\gnl
            \glmpb\gnot{\tau}\gcmpt\grmpb\gcl1\gibr\gnl
            \gcl1\gvac1\gcl1\gibr\gcl1\gnl
            \gcl1\gvac1\gibr\gcl1\gcl1\gnl
            \gend=\gbeg 5{15}
            \gvac1\gvac1\got 3{L}\gnl
            \gvac1\gvac1\gvac1\gcl1\gnl
            \gvac1\gvac1\glmpb\gnot{\iota}\gcmpt\grmpb\gnl
            \gvac1\gvac1\grcm\gcl1\gnl
            \gvac1\gvac1\gibr\gcl1\gnl
            \gvac1\gcn1131\gvac1\gibr\gnl
            \gvac1\gmp{\hat{S}}\gvac1\gcl1\gcl1\gnl
            \glmpb\gnot{\tau}\gcmpt\grmpb\gcl1\gcl1\gnl
            \gcl1\gvac1\gibr\gcl1\gnl
 \gcl1\gvac1\gcl1\gcl1\gcl1\gu1\gnl
            \gend
            $$
            as a result, the map is well defined. Now, we will show the anti-translation indeed determines the inverse of $\mu$. On the one hand, we have
             $$\gbeg 6{16}
            \gvac1\gvac1\gvac1\got 3{L}\gnl
            \gvac1\gvac1\gvac1\gvac1\gcl1\gnl
            \gvac1\gvac1\gvac1\glmpb\gnot{\iota}\gcmpt\grmpb\gnl
            \gvac1\gvac1\gvac1\grcm\gcl1\gnl
            \gvac1\gvac1\gvac1\gibr\gcl1\gnl
            \gvac1\gvac1\gcn1131\gvac1\gibr\gnl
            \gvac1\gvac1\gmp{\hat{S}}\gvac1\gcl1\gcl1\gnl
            \gvac1\glmpb\gnot{\tau}\gcmpt\grmpb\gcl1\gcl1\gnl
            \gvac1\grcm\gibr\gcl1\gnl
            \gcn1131\gcn1131\gvac1\gcl1\gcl1\gcl1\gnl
\gcl1\gdnot\tau\glmptb\grmpb\gcl1\gcl1\gcl1\gnl
\gcl1\gcl1\gcl1\gbr\gcl1\gnl
\gcl1\gcl1\gbr\gbr\gnl
\gcl1\gbr\gbr\gcl1\gnl
\gmu\gbr\gcl1\gcl1\gnl
\gcn2122\gmu\gcl1\gcl1\gnl
            \gend
            =
            \gbeg 6{16}
            \gvac1\gvac1\gvac1\got 3{L}\gnl
            \gvac1\gvac1\gvac1\gvac1\gcl1\gnl
            \gvac1\gvac1\gvac1\glmpb\gnot{\iota}\gcmpt\grmpb\gnl
            \gvac1\gvac1\gvac1\grcm\gcl1\gnl
            \gvac1\gvac1\gvac1\gibr\gcl1\gnl
            \gvac1\gvac1\gcn1131\gvac1\gcl1\gcl1\gnl
            \gvac1\gvac1\gmp{\hat{S}}\gvac1\gcl1\gcl1\gnl
            \gvac1\glmpb\gnot{\tau}\gcmpt\grmpb\gcl1\gcl1\gnl
            \gvac1\grcm\gcl1\gcl1\gcl1\gnl
            \gcn1131\gcn1131\gvac1\gcl1\gcl1\gcl1\gnl
\gcl1\gdnot\tau\glmptb\grmpb\gcl1\gcl1\gcl1\gnl
\gcl1\gcl1\gbr\gcl1\gcl1\gnl
\gcl1\gbr\gbr\gcl1\gnl
\gmu\gbr\gcl1\gcl1\gnl
\gcn2122\gmu\gcl1\gcl1\gnl
            \gend=
            \gbeg 5{13}
            \gvac1\gvac1\got 3{L}\gnl
            \gvac1\gvac1\gvac1\gcl1\gnl
            \gvac1\gvac1\glmpb\gnot{\iota}\gcmpt\grmpb\gnl
            \gvac1\gvac1\grcm\gcl1\gnl
            \gvac1\gvac1\gibr\gcl1\gnl
            \gvac1\gcn1131\gvac1\gcl1\gcl1\gnl
            \gvac1\gmp{\hat{S}}\gvac1\gcl1\gcl1\gnl
            \gvac1\gbbr31\gcl1\gnl
            \gvac1\gcl1\gvac1\gmp{S}\gcl1\gnl
            \gvac1\gcl1\gdnot\tau\glmpb\grmptb\gcl1\gnl
            \gu1\gmu\gcl1\gcl1\gnl
            
            \gend=
            \gbeg 5{9}
            \gvac1\gvac1\got 3{L}\gnl
            \gvac1\gvac1\gvac1\gcl1\gnl
            \gvac1\gvac1\glmpb\gnot{\iota}\gcmpt\grmpb\gnl
            \gvac1\gvac1\grcm\gcl1\gnl
            \gvac1\gcn2131\gcl1\gcl1\gnl
            \gvac1\gcl1\gdnot\tau\glmpb\grmptb\gcl1\gnl
            \gu1\gmu\gcl1\gcl1
            \gend=
            \gbeg 5{6}
            \gvac1\gvac1\got 3{L}\gnl
            \gvac1\gvac1\gvac1\gcl1\gnl
            \gvac1\gvac1\glmpb\gnot{\iota}\gcmpt\grmpb\gnl
            \gu1\gu1\gcl1\gvac1\gcl1
            \gend$$

     where we use (\ref{delli}) and (\ref{gammaeta})  in the second step.       On the other hand, we have
         $$\gbeg8{20}
     \gvac1\gvac1\gvac1\got3{L}\gnl
     \gvac1\gvac1\gvac1\gvac1\gcl1\gnl
     \gvac1\gvac1\gvac1\glmpb\gnot\iota\gcmpt\grmpb\gnl
     \gvac1\gvac1\gvac1\grcm\gcl1\gnl
     \gvac1\gvac1\gvac1\gcl1\gcn1113\gcn2115\gnl
     \gvac1\gvac1\gvac1\gcl1\glmpb\gnot\tau\gcmpt\grmpb\gcn2111\gnl
     \gvac1\gvac1\gvac1\gcl1\gibbr31\gcl1\gnl
     \gvac1\gvac1\gvac1\gibr\gvac1\gibr\gnl
     \gvac1\gvac1\gcn2131\gibbr31\gcl6\gnl
    \gvac1\gvac1\grcm\gcl1\gvac1\gcl5\gnl
            \gvac1\gvac1\gibr\gcl1\gnl
            \gvac1\gcn1131\gvac1\gibr\gnl
            \gvac1\gmp{\hat{S}}\gvac1\gcl1\gcl1\gnl
            \glmpb\gnot{\tau}\gcmpt\grmpb\gcl1\gcl1\gnl
            \gcl1\gvac1\gibr\gbbr31\gcl1\gnl
            \gcl1\gvac1\gcl1\gmu\gvac1\gbr\gnl
            \gcl1\gvac1\gcl1\gcn2122\gvac1\gmu\gnl
    \gend=
    \gbeg8{20}
     \gvac1\gvac1\gvac1\got3{L}\gnl
     \gvac1\gvac1\gvac1\gvac1\gcl1\gnl
     \gvac1\gvac1\gvac1\glmpb\gnot\iota\gcmpt\grmpb\gnl
     \gvac1\gvac1\gvac1\grcm\gcl1\gnl
     \gvac1\gvac1\gvac1\gcl1\gcn1114\gcn2115\gnl
     \gvac1\gvac1\gvac1\gcl1\gvac1\gcmu\gcl1\gnl
     \gvac1\gvac1\gvac1\gcl1\gdnot\tau\glmpb\grmptb\gcl1\gcl1\gnl
     \gvac1\gvac1\gvac1\gcl1\gcl1\gibr\gcl1\gnl
     \gvac1\gvac1\gvac1\gcl1\gibr\gcl1\gcl1\gnl
     \gvac1\gvac1\gvac1\gibr\gmu\gcl1\gnl
     \gvac1\gvac1\gcn2131\gcn1112\gcn2123\gcl1\gnl
     \gvac1\gvac1\gmp{\hat{S}}\gvac1\gcn2122\gcl1\gcl1\gnl 
     \gvac1\glmpb\gnot\tau\gcmpt\grmpb\gcn2122\gibr\gnl
     \gvac1\gcl1\gvac1\gcl1\gcn2121\gcl1\gcl1\gnl
     \gvac1\gcl1\gvac1\gmu\gvac1\gcl1\gcl1\gnl
     \gvac1\gcl1\gvac1\gcn2123\gvac1\gcl1\gcl1\gnl
     \gvac1\gcl1\gvac1\gvac1\gibbr31\gcl1\gnl
     \gend=
         \gbeg6{11}
     \gvac1\gvac1\gvac1\got3{L}\gnl
     \gvac1\gvac1\gvac1\gvac1\gcl1\gnl
     \gvac1\gvac1\gvac1\glmpb\gnot\iota\gcmpt\grmpb\gnl
     \gvac1\gvac1\gvac1\grcm\gcl1\gnl
 \gvac1\gvac1\gvac1\gibr\gcl1\gnl
  \gvac1\gvac1\gcn2131\gcl1\gcl1\gnl
   \gvac1\gvac1\gmp{\hat{S}}\gvac1\gibr\gnl
   \gvac1\glmpb\gnot\tau\gcmpt\grmpb\gcl1\gcl1\gnl
   \gvac1\gcl1\gvac1\gibr\gcl1\gnl
   \gvac1\gcl1\gvac1\gcl1\gmu\gnl
   \gvac1\gvac1\gob3{L\ot_{B}B}
     \gend=
     \gbeg6{11}
     \gvac1\gvac1\gvac1\got3{L}\gnl
     \gvac1\gvac1\gvac1\gvac1\gcl1\gnl
     \gvac1\gvac1\gvac1\glmpb\gnot\iota\gcmpt\grmpb\gnl
     \gvac1\gvac1\gvac1\grcm\gcl1\gnl
 \gvac1\gvac1\gvac1\gibr\gcl1\gnl
  \gvac1\gvac1\gcn2131\gcl1\gcl1\gnl
   \gvac1\gvac1\gmp{\hat{S}}\gvac1\gcl1\gcl1\gnl
   \gvac1\glmpb\gnot\tau\gcmpt\grmpb\gcl1\gcl1\gnl
   \gvac1\gcl1\gvac1\gmu\gcl1\gnl
   \gvac1\gwmuh416\gcl1\gnl
   \gvac1 \gvac1 \gvac1 \gvac1\got1L
     \gend=\gbeg15\got1L\gnl\gcl3\gnl
     \gob1L\gend$$  
     where we use (\ref{gammaeta})  in the second step and 4th step. 
\end{proof}

\subsubsection{In the case that $\B={}_{K}\cM$ for a quasi-triangular Hopf algebra $(K,R)$}
In this section, we consider  Hopf Galois extensions in the braided category of left modules of a quasi-triangular Hopf algebra. 
\label{example of Hopf fibration}

\begin{Lem}\label{lem. ES of a quasi-triangular Hopf algebra}
Let $(K, R)$ be a quasi-triangular Hopf algebra and $H$ be a Hopf algebra with bijective antipode. If $B=P^{co H}\subset P$ is a faithfully flat $H$-Galois extension in ${}_{K}\cM$, then $L(P,H)=(P\ot P)^{co H}$ is a braided $B$-Hopf algebroid in ${}_{(K, R)}\cM$ with the structure:
\[s(b)=b\ot 1,\qquad t(b)=1\ot b,\]
\[(p\ot q)(p'\ot q')=p (R_{\alpha}\la p')\ot (R_{\beta}\la q')((R^{\beta}R^{\alpha})\la q),\]
\[\Delta(p\ot q)=(p\z\ot p\on\tuno{})\ot (p\on\tdue{}\ot q),\qquad \varepsilon(p\ot q)=pq.\]
\[(p\ot q)_{+}\ot(p\ot q)_{-}=(p\ot (R^{\ol \alpha}\,R^{\ol \beta})\la\,q\on\tdue{})\ot (R_{\ol \alpha}\la q\z\ot R_{\ol \beta}\la q\on\tuno{}) \]
\[(p\ot q)_{[+]}\ot(p\ot q)_{[-]}=(S^{-1}(R^{\ol \gamma}\la p\on)\tuno{}\ot (R^{\ol \alpha}R^{\ol \beta})\la q)\ot (R_{\ol \alpha}\la S^{-1}(R^{\ol \gamma}\la p\on)\tdue{}\ot (R_{\ol \beta}R_{\ol \gamma}) \la p\z),\]
for all $p\ot q\in L(P,H)$ and $b\in B$, where the braiding and its inverse are given by
\[\tau_{VW}(v\ot w)=R_{\alpha}\la w\ot R^{\alpha}\la v,\quad\tau^{-1}_{VW}(w\ot v)=R^{\ol \alpha}\la v\ot R_{\ol \alpha}\la w, \]
where $R=R^{\alpha}\ot R_{\alpha}\in K\ot K$ with its inverse $R^{-1}=R^{\ol \alpha}\ot R_{\ol \alpha}\in K\ot K$.
\end{Lem}
\begin{proof} First observe that a Hopf Galois extension in ${}_{K}\cM$ means $H$ is a braided Hopf algebra in ${}_{K}\cM$ and $P$ is a left $K$-module algebra such that 
\begin{align}\label{k equ action}
(k\la p)\z\ot (k\la p)\on=(k\on \la p\z)\ot (k\t\la p\on),
\end{align}
for every $k\in K$ and $p\in P$. 
We can see $B$ is a submodule algebra of $P$. Indeed, by (\ref{k equ action}), $(k\la b)\z\ot (k\la b)\on=(k\on\la b\z)\ot k\t\la b\on=k\la b\ot 1$. 

To check $L(P, H)$ is a left $K$-module, for every $p\ot q\in L(P,H)$, we have
\begin{align*}
\delta(k\la(p\ot q))=&\delta(k\on\la p\ot k\t\la q)\\
=&k\on\la p\z\ot (R_{\alpha}k\th)\la q\z\ot ((R^{\alpha} k\t)\la p\on)\, (k\fo\la q\on)\\
=&k\on\la p\z\ot (k\t R_{\alpha})\la q\z\ot ((k\th R^{\alpha} )\la p\on)\, (k\fo\la q\on)\\
=&k\on\la p\z\ot (k\t R_{\alpha})\la q\z\ot k\th\la  ((R^{\alpha} \la p\on)\, q\on)\\
=&k\la(p\ot q)\ot 1.
\end{align*}
We can also see $L(P, H)\times_{B} L(P, H)$ is also a left $K$-module. Indeed, given $(p\ot q)\ot (p'\ot q')\in L(P, H)\diamond_{B}L(P, H)$, the left $K$-action is given by
\[k\la ((p\ot q)\ot (p'\ot q'))=k\on\la(p\ot q)\ot k\t\la(p'\ot q')=(k\on\la p\ot k\t\la q)\ot (k\th\la p'\ot k\fo\la q').\]
Since the balance tensor product in  $L(P, H)\diamond_{B}L(P, H)$ means 
\[(p\ot qb)\ot (p'\ot q')=(p\ot q)\ot (bp'\ot q'),\]
the left $K$-action on $L(P, H)\diamond_{B}L(P, H)$ is well defined. Now, we  check the action preserves the Takeuchi product. Let $(p\ot q)\ot (p'\ot q')\in L(P, H)\times_{B}L(P, H)$, which is equivalent to the condition that
\[(p\ot q)\ot (p'b\ot q')=(p\ot (R_{\alpha}R_{
\beta
}\la b)(R^{\alpha}\la q))\ot ((R^{\beta}\la p')\ot q'),\]
for all $b\in B$.
We want to check the element under the left $K$-action still satisfies the above condition. Indeed,
\begin{align*}
k\on\la& p\ot k\t\la q\ot (k\th\la p')\,b\ot k\fo\la q'\\
=&k\on\la p\ot k\t\la q\ot k\th\la (p' S(k\fo)\la b)\ot k\fiv\la q'\\
=&k\on\la p\ot k\t\la ((R_{\alpha}R_{
\beta
}S(k\fo))\la b)(R^{\alpha}\la q))\ot (k\th R^{\beta})\la p'\ot k\fiv\la q'\\
=&k\on\la p\ot (k\t R_{\alpha}R_{
\beta
}S(k\fiv)\la b)(k\th R^{\alpha}\la q))\ot (k\fo R^{\beta})\la p'\ot k\si\la q'\\
=&k\on\la p\ot (k\t R_{\alpha}S(k\fiv)\la b)(k\th R^{\alpha}\on\la q))\ot (k\fo R^{\alpha}\t)\la p'\ot k\si\la q'\\
=&k\on\la p\ot (R_{\alpha}k\fo S(k\fiv)\la b)(R^{\alpha}\on k\t\la q))\ot (R^{\alpha}\t k\th)\la p'\ot k\si\la q'\\
=&k\on\la p\ot (R_{\alpha}\la b)(R^{\alpha}\on k\t\la q))\ot (R^{\alpha}\t k\th)\la p'\ot k\fo\la q'\\
=&k\on\la p\ot (R_{\alpha}R_{\beta}\la b)(R^{\alpha} k\t\la q))\ot (R^{\beta} k\th)\la p'\ot k\fo\la q'.
\end{align*}
The rest of the Lemma can be proved by Theorem \ref{thm. braided ES Hopf algebroid}.
\end{proof}

% \begin{Rem}
% Since the translation map is left $K$-linear, we can see the coproduct is also a $K$-module map by using the fact that $k\on\la h\tuno{}\ot k\t\la h\tdue{}=  (k\la h)\tuno{}\ot (k\la h)\tdue{}=\varepsilon(k)h\tuno{}\ot h\tdue{}$. As a result,
% \begin{align*}
% (k\la (p\ot q))\on\ot (k\la (p\ot q))\t=&(k\on \la p\ot k\t\la q)\on\ot(k\on \la p\ot k\t\la q)\t\\
% =&k\on\la p\z\ot p\on\tuno{}\ot p\on\tdue{}\ot k\t\la q\\
% =&k\on\la p\z\ot k\t\la p\on\tuno{}\ot k\th\la p\on\tdue{}\ot k\fo\la q\\
% =&k\la (\Delta(p\ot q)).
% \end{align*}

% \end{Rem}

\section{Braided Lie-Rinehart algebras and their Universal enveloping algebras}

As a generalization of \cite{ALP24,ALP23,ALP25}, we have the definition
\begin{Def}
A braided Lie algebra $\mathfrak{g}$ in $\B$ is a vector space together with a bilinear map $[-,-]:\mathfrak{g}\ot\mathfrak{g}\to \mathfrak{g}$ in $\B$ such that 
\begin{itemize}
\item[(i)] $[-,-]$ is  braided antisymmetry: $[u,v]=-[v_{\alpha},u^{\alpha}]$.
\item[(ii)] $[-,-]$ satisfies braided Jacobi identity: $[u,[v,w]]=[[u,v],w]+[v_{\alpha},[u^{\alpha},w]]$,
\end{itemize}
for all $u,v,w\in\g$.
\end{Def}
A morphism $f:\g\to \h$ between two braided Lie algebras is a morphism in $\B$ such that $f([u,v])=[f(u),f(v)]$. 
Furthermore, we define a braided Lie-Rinehart algebra
\begin{Def}
Given a braided commutative algebra $A$ in $\B$. A braided Lie-Rinehart algebra over $A$ in $\B$ consists of a braided Lie algebra $\g$ and a map $\rho:\g\ot A\to A,\quad u\ot a\mapsto \rho_{u}(a)$ in $\B$ such that 
\begin{itemize}
\item[(i)] $\g$ is a left $A$-module in $\B$,
\item[(ii)] $\rho_{u}(ab)=\rho_{u}(a)b+a_{\alpha}\rho_{u^{\alpha}}(b)$,
\item[(iii)] $\rho_{[u,v]}=\rho_{u}\circ\rho_{v}-\rho_{v_{\alpha}}\circ\rho_{u^{\alpha}}$,
\item[(iv)] $\rho_{au}(b)=a\rho_{u}(b)$,
\item[(v)] $[u,av]=\rho_{u}(a)v+a_{\alpha}[u^{\alpha},v]$,
\end{itemize}
for all $u,v\in\g$ and $a,b\in A$. We call $\rho$ the anchor and we will use $(A,\g,\rho)$ to denote the braided Lie-Rinehart algebra. If $\rho$ is trivial, we call $(A,\g,\rho)$ a braided $A$-Lie algebra. If $\B$ is a symmetric braided category, we call $(A,\g,\rho)$ a symmetric braided Lie-Rinehart algebra. A morphism between two braided Lie-Rinehart algebras is a left $A$-linear braided Lie algebra morphism in $\B$, which preserves the anchor.
\end{Def}

It is similar to the construction of \cite{BKS24,ES19}, we can construct the  universal enveloping algebras of a symmetric braided Lie-Rinehart algebra $(A,\g,\rho)$. First, define 
\[T_{A}(\g\ot A)=\bigoplus_{n\geq 0}(\g\ot A)^{\otimes_{A}^{n}}:=A \oplus(\g\ot A)\oplus((\g\ot A)\ot_{A}(\g\ot A))\oplus\cdots.\]
Let $\eta:\g\to \g\ot A,u\mapsto u\ot 1$. Define a two-sided ideal $\cI$ in $T_{A}(\g\ot A)$ by
\begin{align*}
\cI:=\left\langle
\begin{array}{c}
\eta(u)\otimes_A\eta(v)
-\eta(v_{\alpha})\otimes_A\eta(u^{\alpha})
-\eta\bigl([u,v]\bigr), \\[2mm]
\eta(u)a-a_{\alpha}\eta(u^{\alpha})-\rho_{u}(a)
\end{array}
\ \middle|\
u,v\in\g,\ a\in A
\right\rangle .
\end{align*}

\begin{Lem}
Let $(A,\g,\rho)$ be a symmetric braided Lie-Rinehart algebra $(A,\g,\rho)$.  Then the universal enveloping algebra $\cU_{A}(\g):=T_{A}(\g\ot A)/\cI$ is a braided Hopf algebroid over $A$. More precisely, the source and target map are given by
\[s(a)=t(a)=a,\]
the coproduct and counit can be determined by
\[\Delta(u)=u\ot 1+1\ot u,\qquad \Delta(a)=a\ot 1=1\ot a\]
and
\[\varepsilon(u)=0\qquad \varepsilon(a)=a.\]
The (anti)-translation map is determined  by
\[u_{+}\ot u_{-}=u_{[+]}\ot u_{[-]}=u\ot 1-1\ot u.\]
\end{Lem}
\begin{proof}
First, we observe that the image of $\Delta$ belongs to Takeuchi product. Indeed, on the one hand we have
\begin{align*}
\Delta(u)(a\ot 1)=(u\ot 1+1\ot u)(a\ot 1)=ua\ot1 +a_{\alpha}\ot u^{\alpha}=ua\ot 1+1\ot a_{\alpha}u^{\alpha},
\end{align*}
on the other hand we have
\begin{align*}
\Delta(u)(1\ot a)=(u\ot 1+1\ot u)(1\ot a)=u\ot a+1\ot ua=a^{\ol \alpha}u_{\ol \alpha}+1\ot ua.
\end{align*}
The difference between them is 
\[(ua-a^{\ol \alpha}u_{\ol \alpha})\ot 1-1\ot (ua-a_{\alpha}u^{\alpha})\in \cU_{A}(\g)\diamond_{A}\cU_{A}(\g)\]
since $\B$ is symmetric, it is equal to 
\[\rho_{u}(a)\ot 1-1\ot \rho_{u}(a)=0.\] 
% by setting $\tilde{\Delta}(\eta(u))=\eta(u)\diamond_{A}1+1\diamond_{A}\eta(u)\in T_{A}(\g\ot A)\diamond_{A}T_{A}(\g\ot A)$ and $\tilde{\Delta}(a)=a\diamond_{A} 1\in T_{A}(\g\ot A)\diamond_{A}T_{A}(\g\ot A)$, 
As a result we can define $\Delta(uv)=\Delta(u)\Delta(v)=u\on v\on{}_{\alpha}\ot u\t{}^{\alpha}v\t$. In other words, the product in $\cU_{A}(\g)\diamond_{A}\cU_{A}(\g) $ can be generated by $\Delta(u)$ and $\Delta(a)$ in a braided way. In particular,
\[\Delta(uv)=\Delta(u)\Delta(v)=uv\ot 1+v_{\alpha}\ot u^{\alpha}+u\ot v+1\ot uv\in \cU_{A}(\g)\diamond_{A}\cU_{A}(\g).\]
Now, we can check $\Delta$ factors through $\cI$. Indeed, setting $R(u,v)=u v-v_{\alpha} u^{\alpha}-[u,v]$, we have
\begin{align*}
\Delta&(R(u,v))=\Delta(u v-v_{\alpha} u^{\alpha}-[u,v])\\
=&uv\ot 1+v_{\alpha}\ot u^{\alpha}+u\ot v+1\ot uv\\
&-(v_{\alpha}u^{\alpha}\ot 1+(u^{\alpha})_{\beta}\ot (v_{\alpha})^{\beta}+v_{\alpha}\ot u^{\alpha}+1\ot v_{\alpha}u^{\alpha})-([u,v]\ot1+1\ot [u,v])\\
=&(uv\ot 1-v_{\alpha}u^{\alpha}\ot 1-[u,v]\ot1)+(1\ot uv-1\ot v_{\alpha}u^{\alpha}-1\ot [u,v])\\
=&R(u,v)\ot 1+1\ot R(u,v),
\end{align*}
where we use the fact that $\B$ is symmetric in the second step.
Also, by using $\Delta(a)=a\ot 1$ and setting $R(u,a)=ua-a_{\alpha}u^{\alpha}-\rho_{u}(a)$, we have
\begin{align*}
\Delta&(R(u,a))=\Delta(ua-a_{\alpha}u^{\alpha}-\rho_{u}(a))\\
=&ua\ot 1+a_{\alpha}\ot u^{\alpha} -a_{\alpha}u^{\alpha}\ot 1-a_{\alpha}\ot u^{\alpha}-\rho_{u}(a)\ot 1\\
=&ua\ot 1-a_{\alpha}u^{\alpha}\ot 1-\rho_{u}(a)\ot 1\\
=&R(u,a)\ot 1.
\end{align*}
By the same method, setting $(uv)_{+}\ot (uv)_{-}=u_{+}v_{\alpha+}\ot v_{\alpha-}u^{\alpha}{}_{-}$ and $(uv)_{[+]}\ot (uv)_{[-]}=u_{[+]}v^{\ol \alpha[+]}\ot v^{\ol \alpha[-]}u_{\ol \alpha}{}_{[-]}=u_{[+]}v_{\alpha[+]}\ot v_{\alpha[-]}u^{\alpha}{}_{[-]}$, we have
\[R(u,v)_{+}\ot R(u,v)_{-}=R(u,v)_{[+]}\ot R(u,v)_{[-]}=R(u,v)\ot 1-1\ot R(u,v),\]
\[R(u,a)_{+}\ot R(u,a)_{-}=R(u,a)_{[+]}\ot R(u,a)_{[-]}=R(u,a)\ot 1.\]
Finally, we can check
\begin{align*}
u\on{}_{+}\ot u\on{}_{-}u\t=u\ot1-1\ot u+1\ot u=u\ot1.
\end{align*}
Also,
\begin{align*}
u{}_{+}\on\ot u{}_{+}\t u_{-}=u\ot1+1\ot u-1\ot u=u\ot1.
\end{align*}
Similarly, we can check the anti-translation map.
\end{proof}

\subsection{Braided Lie algebroids of braided Hopf algebroids}
Classically, given a Lie groupoid $\cG$ over its base manifold $M$, the sections of the Lie algebroid (as a vector bundle over $M$) is isomorphic to the right invariant vector fields tangent to the target fiber on the groupoid. The collection of such vector fields together with its Lie bracket and anchor forms a Lie-Rinehart algebra  \cite{Mac}. In this section, we are going to introduce a braided Lie-Rinehart algebra of a braided Hopf algebroid in ${}_{K}\cM$, where $(K, R)$ is a triangular Hopf algebra. We define the (braided) right invariant vector fields on a braided Hopf algebroid. Given a braided $B$-Hopf algebroid $\cL$ in ${}_{K}\cM$, we know $\Hom(\cL,\cL)$ is a left $K$-module algebra with
\[(k\bla u)(X)=k\on\la (u(S(k\t)\la X)).\]
We also observe that $k\la (u(X))=(k\on\bla u)(k\t\la X)$.
We define $\Hom^{\cL}(\cL,\cL)$ by
\[\{u\in \Hom(\cL,\cL)\,|\, u(\ol a X \ol b)=\ol{a_{\alpha}} u^{\alpha}(X)\ol b\quad\textup{and}\quad u(X)\on\ot u(X)\t=u(X\on)\ot X\t\}\]
where $W_{\alpha}\ot V^{\alpha}=R_{\alpha}\la W\ot R^{\alpha}\la V$ and $V_{\ol \alpha}\ot W^{\ol \alpha}=R_{\ol \alpha}\la V\ot R^{\ol \alpha}\la W$ for any element in the objects of the left $K$-module category. We can see $u(X\on)\ot X\t$ is well defined. Indeed, by using the fact that the braided category ${}_{K}\cM$ is symmetric, we have
\begin{align*}
u(\ol{b}X\on)\ot X\t=\ol{b_{\alpha}}u^{\alpha}(X\on)\ot X\t=u(X\on{}_{\alpha})\ot b^{\alpha} X\t.
\end{align*}
We define the braided right invariant vector fields $\cX$ on $\CL$ by
\[\cX:=\{u\in \Hom^{\cL}(\cL,\cL)\,|\, u(XY)=u(X)Y+X_{\alpha}u^{\alpha}(Y)\}\]
We also denote it by $\cX(\cL)$ in case we want to indicate it is the right invariant vector fields on $\cL$. We can see that $u|_{\BB}=0$ by observing the fact that
$b_{\alpha}\,u^{\alpha}(X)=u(\ol{b}X)=u(\ol b)X+b_{\alpha}\,u^{\alpha}(X).$
It is not hard to see, an element $u\in \Hom(\cL,\cL)$ is a braided right invariant vector field if and only if it satisfies
\begin{itemize}
\item[(i)]  $u|_{\BB}=0$;
\item[(ii)]  $u(XY)=u(X)Y+X_{\alpha}u^{\alpha}(Y)$;
\item[(iii)]  $u(X)\on\ot u(X)\t=u(X\on)\ot X\t$.
\end{itemize}
\begin{Lem}
Let $(K,R)$ be a triangular Hopf algebra. If $\cL$ is a braided $B$-Hopf algebroid in ${}_{K}\cM$ then $\cX$ is a symmetric braided Lie algebra in ${}_{K}\cM$ with the Lie bracket $[u, v]=uv-v_{\alpha}u^{\alpha}$.
\end{Lem}
\begin{proof}
First, we show $\cX$ is a left $K$-module. Note that the action given above is well defined. Indeed, given $u\in \cX$, we have $k\bla u(\ol b)=k\on\la (u(\ol{S(k\t)\la b}))=0$. Also, we have
\begin{align*}
(k\bla &u)(XY)\\
=&k\on \la (u((S(k\th)\la X)\,( S(k\t)\la Y)))\\
=&k\on \la (u(S(k\th)\la X)( S(k\t)\la Y))+k\on \la ((R_{\alpha}S(k\th)\la X)(R^{\alpha}\bla u)( S(k\t)\la Y))\\
=&k\on \la (u(S(k\t)\la X))Y+ (k\on R_{\alpha}S(k\th)\la X)(k\t R^{\alpha}\on\la(u(S(R^{\alpha}\t) S(k\th)\la Y)))\\
=&k\on \la (u(S(k\t)\la X))Y+ ( R_{\alpha}k\th S(k\fo)\la X)( R^{\alpha}\on k\on\la(u(S(k\t) S(R^{\alpha}\t) \la Y)))\\
=&k\on \la (u(S(k\t)\la X))Y+ ( R_{\alpha}\la X)( R^{\alpha}\on k\on\la(u(S(k\t) S(R^{\alpha}\t) \la Y)))\\
=&k\bla u(X)\, Y+(R_{\alpha}\la X)(R^{\alpha}k\bla u)(Y).
\end{align*}
Also, we have
\begin{align*}
(k\bla u)(X)\on\ot (k\bla u)(X)\t=&k\on\la u(S(k\th)\la X)\on\ot k\t\la u(S(k\th)\la X)\t\\
=&k\on\la u(S(k\fo)\la X\on)\ot k\t S(k\th)\la X\t\\
=&k\on\la u(S(k\t)\la X\on)\ot X\t\\
=&k\bla u(X\on)\ot X\t.
\end{align*}

Second, we show the Lie bracket is well defined. Indeed, the composition of two vector fields is also right $\cL$-colinear. Moreover, we have
\begin{align*}
(k\on\bla u)(k\t\bla v)(X)=&(k\on\bla u)(k\t\la v(S(k\th)\la X))\\
=&k\on\la u(S(k\t) k\th \la v(S(k\fo)\la X))\\
=&k\on\la u(v(S(k\t)\la X))\\
=&k\bla (uv)(X).
\end{align*}
As a result, $k\bla [u,v]=[k\on\bla u, k\t\bla v]$. 
Indeed,
\begin{align*}
k\bla [u,v]=&k\bla (uv-v_{\alpha}u^{\alpha})=(k\on\bla u)(k\t\bla v)-(k\on\bla v_{\alpha})(k\t\bla u^{\alpha})\\
=&(k\on\bla u)(k\t\bla v)-(k\on R_{\alpha}\bla v)(k\t R^{\alpha}\bla u)\\
=&(k\on\bla u)(k\t\bla v)-( R_{\alpha}k\t\bla v)(R^{\alpha}k\on\bla u)\\
=&[k\on\bla u, k\t\bla v]. 
\end{align*}
It is not hard to see $[u,v]$ is right $\cL$-colinear and vanishes on the target subalgebra. We only check that it satisfies the braided Leibniz rule.
\begin{align*}
[u,v](XY)=&u(v(XY))-v_{\alpha}(u^{\alpha}(XY))\\
=&u(v(X)Y)+u(X_{\alpha}v^{\alpha}(Y))-v_{\alpha}(u^{\alpha}(X)Y)-v_{\alpha}(X_\beta\,u^{\alpha\beta}(Y))\\
=&u(v(X))Y+v(X)_{\alpha}u^{\alpha}(Y)+u(X_\alpha)v^{\alpha}(Y)+X_{\alpha\beta}u^{\beta}(v^{\alpha}(Y))\\
-v_{\alpha}(&u^{\alpha}(X))Y-u^{\alpha}(X)_{\beta}v_{\alpha}{}^{\beta}(Y)-v_{\alpha}(X_{
\beta
})u^{\alpha\beta}(Y)-X_{\beta\gamma
}v_{\alpha}{}^{\gamma}(u^{\alpha\beta}(Y)))\\
=&u(v(X))Y+X_{\alpha\beta}u^{\beta}(v^{\alpha}(Y))-v_{\alpha}(u^{\alpha}(X))Y-X_{\beta\gamma
}v_{\alpha}{}^{\gamma}(u^{\alpha\beta}(Y)))\\
=&[u,v](X)Y+X_{\alpha}[u,v]^{\alpha}(Y)
\end{align*}
where the 3rd step uses the triangularity of $(K,R)$.
It is straightforward to show that $\cX$ satisfies the braided anti-symmetric property and the braided Jacobi identity (the same with proof of \cite[Lemma 5.2]{ALP24}), where we have used the triangularity of $(K,R)$.

\end{proof}
\begin{Thm}Given a triangular Hopf algebra $(K,R)$. If $\cL$ is a braided $B$-Hopf algebroid in ${}_{K}\cM$ and $B,\BB$ belong to the braided center of $\cL$, then $(B, \cX, \rho)$ is a symmetric braided Lie-Rinehart algebra with the anchor given by $\rho_{u}(a)=u(a)$, the left $B$-module structure is given by $(bu)(X)=bu(X)$ for every $a,b\in B$ and the Lie bracket is given by the braided commutator as above.
\end{Thm}
\begin{proof}
First we can see that the left $B$-action is $K$-linear. Indeed,
\[(k\on\la a)(k\t\bla u)(X)=(k\on\la a)( k\t\la u(S(k\th)\la X))=k\on\la (au(S(k\t)\la X))=k\bla (au)(X).\]
Also, we have
\begin{align*}
au(XY)=&au(X)Y+aX_{\alpha}u^{\alpha}(Y)=au(X)Y+X_{\alpha\beta}a^{\beta}u^{\alpha}(Y)\\
=&au(X)Y+X_{\alpha}(au)^{\alpha}(Y),
\end{align*}
where the 2nd step use the fact that $B$ belongs to the braided center. We can also see $au|_{\BB}=0$ and it is clearly right $\cL$-colinear. It is not hard to see
\[\rho_{au}=a\rho_{u},\quad\rho_{[u,v]}=\rho_{u}\rho_{v}-\rho_{v_{\alpha}}\rho_{u^{\alpha}},\quad\rho_{u}(ab)=\rho_{u}(a)b+a_{\alpha}\rho_{u^{\alpha}}(b).\]
Finally, we have
\begin{align*}
[u,av](X)=&u(av(X))-(av)_{\alpha}(u^{\alpha}(X))\\
=&u(a)v(X)+a_{\alpha}u^{\alpha}(v(X))-a_{\alpha}v_{\beta}u^{\alpha\beta}(X).
\end{align*}
As a result, we have $[u,av]=a_{\alpha}[u^{\alpha},v]+\rho_{u}(a)v$.
\end{proof}

\begin{Rem}
 Let $\cL$ be a $B$-Hopf algebroid in ${}_{K}\cM$. 
 Then 
    \begin{align*}
        \Hom_{\BB-}(\cL,B)&\cong \Hom^{\cL}(\cL, \cL)\\
        f&\mapsto (f\diamond_{B} \cL)\circ\Delta=:F_{f}\\
       f_{F}:=\varepsilon\circ F&\mapsfrom F.
    \end{align*}
    Given $F\in \Hom^{\cL}(\cL, \cL)$, we observe that
    \[f_{k\bla F}(X)=\varepsilon(k\on\la F(S(k\t)\la X))=k\on\la\varepsilon( F(S(k\t)\la X))=k\bla f_{F}(X).\]
    It is not hard to see $f_F:\cL\to B$ is braided left $\BB$-linear, i.e. $f_F(\ol b X)=f_{F}(X_{\alpha})b^{\alpha}$. Also, 
    \[F_{f}(\ol b X)=f(X\on{}_{\alpha})\ol{b^{\alpha}}X\t=\ol{b_{\alpha}}f^{\alpha}(X\on)X\t=\ol{b_{\alpha}}F_{f^{\alpha}}(X)=\ol{b_{\alpha}}F_{f}^{\alpha}(X).\]
   
    Moreover, we can check that 
    \[F_{f_{F}}(X)=f_{F}(X\on)X\t=\varepsilon(F(X\on))X\t=\varepsilon(F(X)\on)F(X)\t=F(X).\]
   Also,
   \[f_{F_{f}}(X)=\varepsilon(F_{f}(X))=\varepsilon(f(X\on)X\t)=f(X\on)\varepsilon(X\t)=f(X).\]
  Given $u\in \cX$, we define $\Tilde{u}:=\varepsilon\circ u\in \Hom_{\BB-}(\cL,B)$.   The geometric meaning of $\Tilde{u}$ is the restriction of the right invariant vector field $u$ on the unit submanifold $M$, which is a section of the Lie algebroid vector bundle over the base manifold.
\end{Rem}

With the same assumption as above, we can also define 
\[\cX_{\textup{ver}}:=\{u\in \cX\,|\, u(b)=0,\forall b\in B\}\]
We have
\begin{Lem}
Given a triangular Hopf algebra $(K,R)$. If $\cL$ is a braided $B$-Hopf algebroid in ${}_{K}\cM$ and $B,\BB$ belong to the braided center of $\cL$ then $(B, \cX_{\textup{ver}}, \rho)$ is a symmetric braided $B$-Lie algebra as a sub-Lie-Rinehart algebra of $(B, \cX, \rho)$.
\end{Lem}
\begin{proof}
First, $\cX_{\textup{ver}}$ is a $K$-module. Indeed,
\[k\bla u(b)=k\on\la(u(S(k\t)\la b))=0.\]
Clearly, the Lie bracket of two elements in $\cX_{\textup{ver}}$ also vanishes on $B$. Moreover, $\cX_{\textup{ver}}$ is a left $B$-module and the anchor $\rho$ is trivial.
\end{proof}
There is a short exact sequence of Lie-Rinehart algebras:

 \[0\xrightarrow{}\cX_{\textup{ver}}\xrightarrow{}\cX\xrightarrow{\rho}\textup{Der}(B)\xrightarrow{}0,
           \]
where $\textup{Der}(B):=\{u:B\to B| u(ab)=u(a)b+a_{\alpha}u^{\alpha}(b)\}$ is the braided derivative of $B$ which is clearly a Lie-Rinehart algebra by the same reason as above.

\subsubsection{Braided Lie algebroids of  braided ES Hopf algebroids}
In this subsection, we will study the braided Lie-Rinehart algebra of ES-Hopf algebroid. 
Given a faithfully flat $H$-Galois extension $B\subset P$ in ${}_{K}\cM$ as in Lemma \ref{lem. ES of a quasi-triangular Hopf algebra}. It is similar to \cite{ALP24,ALP23, ALP25}, we define
\[\textup{Der}^{H}(P):=\{f\in \textup{Der}(P)\,|\, f(p)\z\ot f(p)\on=f(p\on)\ot p\z\}.\]
\begin{Thm}
Let $(K,R)$ be a triangular Hopf algebra and $B\subset P$ be a faithfully flat $H$-Galois extension in ${}_{K}\cM$ such that $B$ belongs to the braided center of $P$. Then $\textup{Der}^{H}(P)$ is a symmetric braided Lie-Rinehart algebra over $B$ which is isomorphic to $\cX(L(P, H))$. More precisely, the structure is given by
\[k\bla f(p)=k\on\la (f(S(k\t)\la p));\]
\[[f,g]=f\circ g-g_{\alpha}\circ f^{\alpha};\]
\[(bf)(p)=bf(p),\quad \rho_{f}(b)=f(b).\]
\end{Thm}
\begin{proof}
It is sufficient to show there is a one to one correspondence between $\cX(L(P, H))$ and $\textup{Der}^{H}(P)$ which preserve all the structure above. 

Given $f\in \textup{Der}^{H}(P)$, we can construct $u_{f}(p\ot q)=f(p)\ot q$ for every $p\ot q\in L(P, H)$. As $f$ is right $H$-colinear, we can see $f(p)\ot q\in L(P,H)$. We also have
\begin{align*}
u_{f}((p\ot q)(p'\ot q'))=&u_{f}(pp'_{\alpha}\ot q'_{\beta}q^{\alpha\beta})=f(pp'_{\alpha})\ot q'_{\beta}q^{\alpha\beta}\\
=&f(p)p'_{\alpha}\ot q'_{\beta}q^{\alpha\beta}+p_{\gamma}f^{\gamma}(p'_{\alpha})\ot q'_{\beta}q^{\alpha\beta}\\
=&u_{f}(p\ot q)(p'\ot q')+(p\ot q)_{\gamma}u_{f^{\gamma}}(p'\ot q')\\
=&u_{f}(p\ot q)(p'\ot q')+(p\ot q)_{\gamma}(u_{f})^{\gamma}(p'\ot q')
\end{align*}
where we use the  triangularity of $(K,R)$ in the 3rd step and the fact that
\begin{align*}
k\bla u_{f}(p\ot q)=&k\on \la (f(S(k\t)\la (p\ot q)))=k\on\la (f(S(k\th)\la p)\ot S(k\t)\la q)\\
=&k\on\la f(S(k\t)\la p)\ot  q=u_{k\bla f}(p\ot q)
\end{align*}
in the 4th step. Also, $u_{f}(\ol b)=f(1)\ot b=0$. Moreover,
\begin{align*}
u_{f}(p\ot q)\on\ot u_{f}(p\ot q)\t
=&f(p\z)\ot p\on\tuno{}\ot p\on\tdue{}\ot q\\
=&u_{f}((p\ot q)\on)\ot (p\ot q)\t.
\end{align*}
Hence $u_{f}$ is a well defined right invariant vector field.

Given $u\in \cX$, we can construct $f_{u}\in\textup{Der}^{H}(P)$ by $f_{u}(p)=\varepsilon\circ u(p\mo)p\z=\varepsilon\circ u(p\z\ot p\on\tuno{})p\on\tdue{}$. As $u$ is $\BB$-linear, this is well defined by using $p\z\ot p\on\tuno{}\ot p\on\tdue{}\in L(P,H) \diamond_{B}P$. It is right $H$-colinear by using (\ref{delre}). Now, we check it is a braided derivative. 
\begin{align*}
f_{u}(p q)=&\varepsilon\circ ((pq)\mo)(pq)\z\\
=&\varepsilon\circ u(p\mo q\mo{}_{\alpha})p\z{}^{\alpha}q\z\\
=&\varepsilon(u(p\mo) q\mo{}_{\alpha})p\z{}^{\alpha}q\z+\varepsilon(p\mo{}_{\beta} u^{\beta}(q\mo{}_{\alpha}))p\z{}^{\alpha}q\z\\
=&\varepsilon(u(p\mo) \ol{\varepsilon(q\mo{}_{\alpha})})p\z{}^{\alpha}q\z+\varepsilon(p\mo{}_{\beta} \ol{\varepsilon (u^{\beta}(q\mo{}_{\alpha}))})p\z{}^{\alpha}q\z\\
=&\varepsilon(u(p\mo) \ol{\varepsilon(q\mo{}_{\alpha})})p\z{}^{\alpha}q\z+\varepsilon(p\mo{}_{\beta} \ol{\varepsilon (u^{\beta}(q\mo{}^{\ol \alpha}))})p\z{}_{\ol \alpha}q\z\\
=&\varepsilon(u(p\mo) )p\z{}\varepsilon(q\mo{})q\z+\varepsilon(p{}_{\beta}\mo)p_{\beta}\z \varepsilon(u^{\beta}(q\mo))q\z\\
=&f_{u}(p)q+p_{\alpha}(f_{u})^{\alpha}(q).
\end{align*}
where the first step uses  the left $L(P, H)$ coaction on $P$ is an algebra map. The 5th step uses the  triangularity of $(K,R)$. The sixth step uses the graphical computation that
$$\gbeg6{7}
    \got1{\cX}\gvac1\got1{P}\gvac1\got1{P}\gnl
    \gcl1\glcm\glcm\gnl
    \gbr\gibr\gcl1\gnl
    \gcl1\glm\gcl1\gcl1\gnl
    \gcn2113\gmp{\ol \varepsilon}\gcl1\gcl1\gnl
    \gvac1\grm\gcl1\gcl1\gnl
    \gvac1\gob1{L}\gvac1\gob1{P}\gob1{P}
   \gend=\gbeg6{7}
    \gvac1\got1{\cX}\got1{P}\gvac1\got1{P}\gnl
    \gvac1\gbr\glcm\gnl
    \glcm\glm\gcl1\gnl
    \gcl1\gcl1\gvac1\gmp{\varepsilon}\gcl1\gnl
    \gcl1\gcn2113\gcl1\gcl1\gnl
    \gcl1\gvac1\grm\gcl1\gnl
    \gob1{L}\gvac1\gob1{P}\gvac1\gob1{P}
   \gend$$
   where we also use the triangularity of $(K,R)$.
Now, we show the correspondence is bijective. Indeed, on the one hand,
\begin{align*}
u_{f_{u}}(p\ot q)=&f_{u}(p)\ot q=\varepsilon\circ u(p\z\ot p\on\tuno{})p\on\tdue{}\ot q\\
=&\varepsilon\circ u((p\ot q)\on)(p\ot q)\t=\varepsilon( u(p\ot q)\on) u(p\ot q)\t=u(p\ot q).
\end{align*}
On the other hand,
\begin{align*}
f_{u_{f}}(p)=\varepsilon\circ u_{f}(p\z\ot p\on\tuno{})p\on\tdue{}=f(p\z)p\on\tuno{}p\on\tdue{}=f(p).
\end{align*}
We can also check the correspondence preserves all the Lie-Rinehart algebra structure. We already checked it is $K$-linear. Also, we have
\[au_{f}(p\ot q)=af(p)\ot q=u_{af}(p\ot 1).\]
It preserves the anchor
\[\rho_{u_{f}}(a)=u_{f}(a\ot 1)=f(a)=\rho_{f}(a).\]
It preserves the Lie bracket
\[u_{[f,g]}(p\ot q)=[f,g](p)\ot q=(f(g(p))-g_{\alpha}(f^{\alpha}(p)))\ot q=(u_{f}u_{g}-u_{g_{\alpha}}u_{f^{\alpha}})(p\ot q).\]
\end{proof}

Define 
\[\textup{aut}_{B}(P):=\{f\in \textup{Der}^{H}(P)\,|\, f(b)=0,\forall b\in B\}.\]
We have
\begin{Cor}
Let $(K,R)$ be a triangular Hopf algebra and $B\subset P$ be a faithfully flat $H$-Galois extension in ${}_{K}\cM$ such that $B$ belongs to the braided center of $P$. Then we have the isomorphism of the short exact sequence of Lie-Rinehart algebra in ${}_{K}\cM$:
  \[\xymatrix{0\ar[r]&\cX_{\textup{ver}}(L(P,H))\ar[r]^-{\iota}\ar[d]_{}&\cX(L(P,H))\ar[r]^{\rho}\ar[d]_{} &\textup{Der}(B)\ar[d]_{=}\ar[r]&0\\
0\ar[r]&\textup{aut}_{B}(P)\ar[r]^-{\iota} &\textup{Der}^{H}(P)\ar[r]^{\rho}&\textup{Der}(B)\ar[r]&0},\]
\end{Cor}
\begin{proof}
By the above Lemma, given $u\in \cX_{\textup{ver}}(L(P,H))$, we have $f_{u}(b)=u(b\ot 1)=0$. Conversely, given $f\in \textup{aut}_{B}(P)$, we have $u_{f}(b\ot 1)=f(b)=0$.
\end{proof}

\section{Braided Jet Hopf algebroids}

As a generalization of jet Hopf algebroids associated to pair Hopf algebroid \cite{HM25}, we will introduce braided jet Hopf algebroids $\cL$  associated to braided Hopf algebroid whose base algebra $B$ and its opposite are braided center of the Hopf algebroid, i.e. $bX=X_{\alpha}b^{\alpha}$ and $\ol b\, X=X_{\alpha}\ol{b^{\alpha}} $ for all $X\in \cL$ and $b\in B$. We define 
\[\mu^{k+1}:=\cL\ker(\varepsilon)^{k+1}.\]
We call $\cJ^{k}(\cL):=\cL/\mu^{k+1}$ the $k$-th jet bundle of $\cL$. Classically, this is the $k$-th jet bundle of the Lie algebroid of a Lie groupoid \cite{LSX}.
We observe that $\cdots\mu^k\subseteq\cdots \subseteq\mu^2\subseteq \mu^1$. We call $\{\mu^{k}\}_{k\in \N}$ stabilize if there is a $n\in N$, such that $\mu^n=\mu^{n+1}$. In this case we denote $\mu^{\infty}=\mu^n$.

\begin{LemDef}\label{Lem jet Hopf algebroid}
Let $\cL$ be a  braided Hopf algebroid over $B$ and faithfully flat on both sides over $B$ and $\BB$. If $B$ and its opposite algebra belong to the braided center of $\cL$ and $\{\mu^{k}\}_{k\in \N}$ stabilizes, then $\cJ(\cL):=\cL/\mu^{\infty}$ is a braided Hopf algebroid with the structure induced from $\cL$. We call $\cJ(\cL)$ the jet Hopf algebroid of $\cL$.
\end{LemDef}
\begin{proof}
Since $B$ is a braided center of $\cL$, $\mu^\infty$ is a two-sided ideal of $\cL$. Also, we have
\[\Delta(\mu)\subset \ker(\varepsilon\diamond_{B}\varepsilon)=\ker(\varepsilon)\diamond_{B}\cL+\cL\diamond_{B}\ker(\varepsilon).\]
Indeed, let $X\in \mu$, we have
\begin{align*}
(\varepsilon\diamond_{B}\varepsilon)(X\on\ot X\t)=\varepsilon(X\on)\ot \varepsilon(X\t)=1\ot \varepsilon(\varepsilon(X\on)\,X\t)=1\ot\varepsilon(X)=0.
\end{align*}
As a result, $\mu^\infty$ is a coideal since $\Delta$ is an algebra map.
To check $\mu^\infty$ is a Hopf ideal, we need the fact that the source and target of the base belong to the center of $\cL$.
We can see that $\varepsilon$ is right $B$-linear. Indeed,
\begin{align*}
\varepsilon(Xb)=\varepsilon(b^{\ol \alpha}\,X_{\ol \alpha})=b^{\ol \alpha}\varepsilon(X_{\ol \alpha})=\varepsilon(X)\,b.
\end{align*}
Hence, $\varepsilon\ot_{B}\varepsilon:\cL\ot_{B}\cL\to B$ is well defined. Also, we have
\[\ker(\varepsilon\ot_{B}\varepsilon)=\ker(\varepsilon)\ot_{B}\cL+\cL\ot_{B}\ker(\varepsilon).\]
Let $X\in \mu$, 
\[\varepsilon(X_{[+]})\ot_{B}\varepsilon(X_{[-]})=\varepsilon(X_{[+]}\varepsilon(X_{[-]}))\ot_{B}1=\varepsilon(X)\ot_{B} 1=0.\]
As a result, $\mu^\infty$ is an anti-Hopf ideal. By the same method, we have $\varepsilon(X\ol b)=\varepsilon(X)\ol b$, hence 
\begin{align*}
\ol {\varepsilon(X_{+})}\ot_{\BB}\ol {\varepsilon(X_{-})}=&\ol {\varepsilon(X_{+})}\,\ol {\varepsilon(X_{-})}\ot_{\BB}1=\ol{\varepsilon(X_{-}{}_{\alpha})\varepsilon(X_{+}{}^{\alpha})}\ot_{\BB}1\\
=&\ol{\varepsilon(X_{+})\varepsilon(X_{-})}\ot_{\BB}1
=\ol{\varepsilon(X_{+}\ol{\varepsilon(X_{-})})}\ot_{\BB}1=0
\end{align*}
\end{proof}

\begin{Rem}
It is similar to \cite{HM25} that we can also define braided pro-Hopf algebroids in $\B$, in which case all the pro-objects are in the category associated to the braided category. For example, a pro-$B$-coring is defined in terms of a pro-object in the $B$-bimodule category in $\B$. By the same method as in \cite[Thm. 4.10]{HM25}, without the assumption that $\{\mu^k\}_{k\in \N}$ stabilizes, we can still show the $\varprojlim_{i}\cJ^{i}(\cL)=\varprojlim_{i}\cL/\mu^{i}$ is a braided pro-Hopf algebroid.
\end{Rem}

It is known that the $k$-th order differential operators is the dual space of the $k$-th jet space (even in the noncommutative case \cite{HM25}).
Recall that there is a  filtration
\[A\subset \cU_{A}(\g)_{\leq 1}\subset \cU_{A}(\g)_{\leq 2}\cdots\subset \cU_{A}(\g).\]
Here we are going to show the $\cU_{B}(\g)_{\leq k}$ is the dual space of the $k$-th jet bundle of a braided Lie-algebroid. 

\begin{Thm}
Let $(K,R)$ be a triangular Hopf algebra, $\cL$ be a faithfully flat (on both sides over its base and its opposite algebra) Hopf algebroid over $B$ in ${}_{K}\cM$, $B$ and $\BB$ belong to the braided center of $\cL$. Then $\cU_{B}(\cX(\cL))_{\leq j}$ belongs to the dual space of $\cJ^{j}(\cL)$. In addition, if $\{\mu^{j}\}_{j\in \N}$ stabilizes then there is a skew pairing between $\cU_{B}(\cX(\cL))$ and $\cJ(\cL)$.
\end{Thm}
\begin{proof}
Let $j=0$. We know $\cJ^{0}(\cL)=\cL/\mu$ and $\cU_{B}(\cX)_{\leq 0}=B$. The dual pairing is given by $\skp{b}{[X]}=b\varepsilon(X)$. Let $j=1$, we know $\cJ^{1}(\cL)=\cL/\mu^2$ and $\cU_{B}(\cX)_{\leq 1}=(B\oplus \cX\ot B)/\cI$. The dual pairing is given by $\skp{u\ot a}{[X]}:=\varepsilon\circ u(aX)$. It is not hard to see this is braided left $\BB$-linear. Moreover, it factors through $\mu^2$. Indeed, let $i,j\in \mu$, we have $\varepsilon(u(ij))=\varepsilon(u(i)j)+\varepsilon(i_{\alpha}u^{\alpha}(j))=\varepsilon(u(i)\varepsilon(j))+\varepsilon(i_{\alpha}u^{\alpha}(j))=\varepsilon(u(j_{\alpha})i^{\alpha})=\varepsilon(u(j_{\alpha})\varepsilon(i^{\alpha}))=0$. Moreover, 
\[\skp{u\ot a-a_{\alpha}u^{\alpha}-\rho_{u}(a)}{[X]}=\varepsilon(u(aX)-a_{\alpha}u^{\alpha}(X)-u(a)X)=0.\]
Let $j=2$, $\cU_{B}(\g)_{\leq 2}=(B\oplus \cX\ot B\oplus(\cX\ot B)\ot_{B}(\cX\ot B))/\cI$. The dual pairing is given by $\skp{u v}{[X]}:=\varepsilon(u\circ v(X))$. By using the braided Leibniz rule, we can also see $\varepsilon(u\circ v(ijk))=0$ for all $i,j,k\in \mu$. 
Moreover, by the definition of the braided Lie bracket
\[\skp{u v-v_{\alpha} u^{\alpha}-[u,v]}{[X]}=0.\]
For higher order $j$, observe that given $u_{1}\cdots u_{j}\in \cX$ and $i_1,\cdots i_{j+1}\in \mu$. For 
\[u_{1}\circ,\cdots \circ u_{j}(i_1\cdots i_{j+1}),\]
by using the braided Leibniz rule, for each summand, there is at least one element never be acted by a right invariant vector field. As a result, the skew pairing given by
\[\skp{u_{1}\cdots u_{j}}{[X]}=\varepsilon(u_{1}\circ\cdots \circ u_{j}(X))\]
factors through $\mu^{j+1}$. 

 If $\{\mu^{j}\}_{j\in \N}$ stabilizes, we can define a skew pairing as above.
First, it is not hard to see
\[\skp{a ub}{[X_{\alpha}]}e^{\alpha}=\varepsilon(a u(b\ol{e}X))=a\skp{u}{b\ol{e}[X]}.\]
Also, we have
\[\skp{u\ol{b}}{[X]}=\varepsilon( u(b X))=\varepsilon( u(X_{\alpha}b^{\alpha}))=\skp{u}{[X^{\ol \alpha}]b_{\ol \alpha}},\]
where we use $B$ belongs to the braided center in the last step.
Moreover, we have
\[\skp{\ol{b}u}{[X]}=b\varepsilon( u( X))=\varepsilon( u_{\alpha}(X_{\beta}))b^{\alpha\beta}=\varepsilon( u_{\alpha}(\ol{b^{\alpha}}X))=\varepsilon( u_{\alpha}(X_{\beta}\ol{b^{\alpha\beta}}))=\skp{u^{\ol{\alpha}}}{[X^{\ol \beta}]\ol{b_{\ol {\alpha}\ol{\beta}}}},\]
where we use $B$ belongs to the braided center in the 4th step.
For the product  properties, we have on the one hand
\[\skp{u}{[X][Y]}=\varepsilon(u(XY))=\varepsilon(u(X)Y)+\varepsilon(X_{\alpha}u^{\alpha}(Y))\]
On the other hand
\begin{align*}
\skp{u\on}{[X^{\ol \alpha}]\ol{\skp{u\t{}_{\ol \alpha}}{[Y]}}}
=&\skp{u}{[X]\ol{\skp{1}{[Y]}}}+\skp{1}{[X^{\ol \alpha}]\ol{\skp{u{}_{\ol \alpha}}{[Y]}}}\\
=&\varepsilon(u(X)Y)+\varepsilon(X_{\alpha}u^{\alpha}(Y)),
\end{align*}
where we use the triangularity in the second step. For the  coproduct properties, we have on the one hand
\[\skp{uv}{[X]}=\varepsilon(u(v(X))).\]
On the other hand,
\begin{align*}
\skp{u}{\skp{v}{[X\on]}[X\t]}=&\varepsilon(u(\varepsilon(v(X\on))X\t))=\varepsilon(u(\varepsilon(v(X)\on)v(X)\t))\\
=&\varepsilon(u(v(X))).
\end{align*}
\end{proof}

\subsection{Examples on Hopf fibrations and homogeneous spaces}If $H$ is a Hopf algebra with trivial $K$-action, i.e. $k\la h=\varepsilon(k)h$, for all $k\in K$ and $h\in H$. Then the coaction is a morphism in ${}_{K}\cM$ means
\begin{align}\label{k equ action1}
(k\la p)\z\ot (k\la p)\on=(k \la p\z)\ot p\on.
\end{align}
There are two examples in  \cite{ALP24,ALP23, ALP25, HLL23, LS04}, $\theta$-deformation Hopf fibration $\cO(S_{\theta}^4)\subseteq \cO(S_{\theta}^7)$ is a $U(\mathfrak{t}^2)$-braided $\cO(SU(2))$-Galois extension (i.e. the corresponding braided category is the left $U(\mathfrak{t}^2)$-module). Also, $\cO(S_{\theta}^{2n})\subseteq \cO(SO_{\theta}(2n+1,\mathbb{R}))$ is a $K$-braided $\cO(SO_{\theta}(2n,\mathbb{R}))$-Galois extension with $K=U(\mathfrak{t}^n)^{op}\ot U(\mathfrak{t}^n)$. Moreover, for both of the examples, the base algebra $B$ is the braided center of the total space algebra $P$, as a result, $\varprojlim_{i}\cJ^{i}(L(\cO(SO_{\theta}(2n+1,\mathbb{R})), \cO(SO_{\theta}(2n,\mathbb{R}))))$ and  $\varprojlim_{i}\cJ^{i}(L(\cO(S_{\theta}^7), \cO(SU(2))))$ are (pro)-braided jet Hopf algebroids. Moreover, $\cU_{B}(\cX(L(P,H)))_{\leq j}$ belongs to the dual space of $\cJ^{j}(L(P,H))$.

\bigskip
\begin{flushleft}
\small
Department of Mathematics, Queen Mary University of London, London, UK\\
\textit{E-mail address:} \texttt{x.h.han@qmul.ac.uk}
\end{flushleft}
\end{document}